\documentclass[a4paper,english,leqno]{amsart}

\usepackage[utf8]{inputenc}
\usepackage{mathrsfs}
\usepackage{babel}
\usepackage{amsmath,amsthm,amssymb,xy,amsxtra,latexsym,verbatim}
\usepackage{hyperref}
\usepackage{enumitem}
\usepackage{pdfsync}
\usepackage{url}

\usepackage{amsxtra}
\usepackage{mathabx}
\usepackage{pxfonts}

\newcommand*{\ii}{\mathrm{i}}

\newcommand*{\scrL}{\ensuremath{\mathscr{L}}}	

\newcommand*{\caH}{\ensuremath{\mathcal{H}}}
\newcommand*{\caI}{\ensuremath{\mathcal{I}}}
\newcommand*{\caS}{\ensuremath{\mathcal{S}}}		
\newcommand*{\caL}{\ensuremath{\mathcal{L}}}		
\newcommand*{\caO}{\ensuremath{\mathcal{O}}}		
\newcommand*{\caT}{\ensuremath{\mathcal{T}}}		

\newcommand*{\bP}{\mathbf{P}}					
\newcommand*{\bL}{\mathbf{L}} 					
\newcommand*{\bF}{\mathbf{F}} 					

\newcommand*{\wtC}{\widetilde{C}}

\newcommand*{\frS}{\mathfrak{S}}

\def\ds{\displaystyle}

\newcommand{\Hzzeta}{\mathcal{H}_{z,\zeta}}
\newcommand{\HzzetaRd}{\mathcal{H}_{z,\zeta}(\R^d)}
\newcommand{\Lip}{\operatorname{Lip}^{-1,-r}}
\newcommand{\Liploc}{\Lip_\mathrm{loc}}

\newcommand*{\N}{\mathbb{N}}										
\newcommand*{\Z}{\mathbb{Z}}										
\newcommand*{\R}{\mathbb{R}}										

\newcommand{\x}{\langle x\rangle}
\newcommand{\csi}{\langle \xi \rangle}

\newcommand{\pdd}{\langle D \rangle}
\newcommand{\jap}{\langle \cdot\rangle}
\def\<{{\langle}}
\def\>{{\rangle}}

\numberwithin{equation}{section}
\allowdisplaybreaks

\theoremstyle{plain}
\newtheorem{lemma}{Lemma}[section]
\newtheorem{theorem}[lemma]{Theorem}
\newtheorem{proposition}[lemma]{Proposition}

\newtheorem{assumption}[lemma]{Assumptions}

\theoremstyle{definition}
\newtheorem{definition}[lemma]{Definition}
\newtheorem{remark}[lemma]{Remark}

\newcommand{\beqsn}{\arraycolsep1.5pt\begin{eqnarray*}}
\newcommand{\eeqsn}{\end{eqnarray*}\arraycolsep5pt}
\newcommand{\beqs}{\arraycolsep1.5pt\begin{eqnarray}}
\newcommand{\eeqs}{\end{eqnarray}\arraycolsep5pt}

\usepackage{xcolor}
\definecolor{red}{rgb}{1,0,0}

\def\Op{ {\operatorname{Op}} }

\def\minull{(0,0,\cdots)}

\begin{document}
	
	\title[Chaos expansion solutions of magnetic Schr\"odinger Wick-type SPDEs on $\R^d$]
	{Chaos expansion solutions of\\a class of magnetic Schr\"odinger Wick-type stochastic equations on $\R^d$}
	
	\author{Sandro Coriasco}
	\address{Dipartimento di Matematica ``G. Peano'', Universit\`a degli Studi di Torino, via Carlo Alberto n.~10, I-10123 Torino, Italy}
	\email{sandro.coriasco@unito.it}
	
	\author{Stevan Pilipovi\'c}
	\address{Department of Mathematics and Informatics, Faculty of Sciences,
		University of Novi Sad, Trg D. Obradovi\'ca 4, 
		RS-21000 Novi Sad, Serbia}
	\email{pilipovic@dmi.uns.ac.rs}
	
	\author{Dora Sele\v{s}i}
	\address{Department of Mathematics and Informatics, Faculty of Sciences,
		University of Novi Sad, Trg D. Obradovi\'ca 4, 
		RS-21000 Novi Sad, Serbia}
	\email{dora@dmi.uns.ac.rs}
	
	\date{}

\begin{abstract}
	We treat some classes of linear and semilinear stochastic partial differential equations of Schr\"odinger type on $\R^d$,
	involving a non-flat Laplacian, within the framework of white noise analysis,
	combined with Wiener-It\^o chaos expansions and pseudodifferential operator methods. The initial data and potential term of the Schr\"odinger operator are assumed to be generalized stochastic processes that have 
	 spatial dependence.	We prove that the equations under consideration have unique solutions in the appropriate (intersections of
	weighted) Sobolev-Kato-Kondratiev spaces. 
\end{abstract}



\maketitle

\tableofcontents

%
\section{Introduction}\label{sec:intro}

The Schr\"odinger equation lies at the heart of quantum mechanics, providing a fundamental framework for describing the behavior and evolution of quantum systems. In many real-world scenarios, quantum systems are subject to environmental fluctuations and stochastic influences, which necessitate the development of advanced mathematical tools to accurately model their dynamics. The stochastic Schr\"odinger equation is a powerful extension of the Schr\"odinger equation that takes into account random elements 
(for instance, fluctuations and uncertainties can be  incorporated into the equation via white noise or other singular generalized stochastic processes), enabling a more 
comprehensive representation of quantum dynamics in stochastic environments. By combining stochastic analysis with pseudodifferential calculus, we develop a robust 
mathematical framework, capable of addressing quantum systems, influenced by highly singular, fluctuating and unpredictable factors. 

In this paper we focus on Cauchy problems associated with Schr\"odinger type differential operators, 
allowing random terms to be present both in the initial conditions, as well as in the potential term of the involved operators,
and we aim at working within the environment of generalized functions. 
Having all these highly random terms leads to singular solutions that do not allow to use ordinary multiplication. A widely employed approach 
to overcome this difficulty consists in its renormalization, also known as the so-called Wick product. 
%
The Wick product is known to represent the highest order stochastic approximation of the ordinary product \cite{Mikulevicius}, and has been used in many models together with the Wiener chaos expansion method, see, e.g., 
\cite{Hida, HOUZ, Milica, Milica2, Boris, BorisBook, GRPW, ps, DP2, DoraS, SM23}. 
By replacing ordinary products, the Wick product helps regularizing singularities in the equation,
ensuring that the solutions remain well-defined, even in the presence of singularities that make an ordinary product between stochastic processes impossible. 
This is related to the celebrated impossibility result of Schwartz in the deterministic case, that makes higher powers of a Dirac delta distribution not possible within 
linear distribution theory. 

The Wick product also involves integrating over all possible outcomes or sample paths of the
underlying stochastic processes. This integration captures the combined influence of random variables
across the entire sample space, rather than focusing only on individual outcomes or pointwise interactions. Similarly as convolution integrals capture 
the influence of past states or trajectories on current behavior as a "memory effect" (e.g., fractional derivatives in applications), 
the Wick product can be viewed to capture the joint influence of random variables on the overall system dynamics, integrating the collective
behavior of stochastic processes across all possible outcomes. One important consequence of using the Wick product is the unbiasedness of the solution to the model 
SPDE: the expected value of the SPDE is equal to the solution of the SPDE with no noise (in our case, the zeroth coefficient in the chaos expansion).

Through this approach we aim to pave the way for further studies in various noisy and fluctuating settings. 
In particular, the magnetic Schr\"odinger type operators that here we study on $\R^d$ could be considered also on the wider setting of suitable classes of
non-compact Riemannian manifolds as spatial domains (see, e.g., \cite{CD21, Krainer, ME}). This could open up new avenues of exploration, for instance
comparing our results with those coming from the algebraic and microlocal approach to SPDEs, cf. \cite{BDR23}, \cite[Sections 1.1 and 1.2]{DDRZ22} (under suitable
hypotheses \textit{at infinity} on the non-compact base manifold $M$, and working with the analog of tempered distributions on it), 
or either in areas where curved, exotic geometries play a relevant role, such as metamaterial design, cf. \cite{GCVetal22, NYKA22},
or in manipulating electromagnetic waves at the nanoscale, cf. \cite{CSHZCXP25, VGPZ17}. 

\medskip

In recent years, pseudodifferential operators have emerged as a valuable mathematical tool in the study of partial differential equations and their stochastic counterparts, leading to an even more rapid development in this area (see, for instance, \cite{ACS19b, ACS19a, linearpara, ACS19c, semilinearpara, alessiandre} 
and the references quoted therein). Pseudodifferential operators extend the concept of ordinary differential operators, enabling the analysis and manipulation of functions that exhibit singular behavior. By employing pseudodifferential operators onto singular input data, in our setting, on symbol classes satisfying global estimates on the whole phase-space $\R^d\times\R^d$ (see, e.g., \cite{cordes}), combined with the chaos expansion methods from stochastic analysis, we can address the challenges posed by both singularity and stochasticity and capture the intricate interplay between quantum mechanics, pseudodifferential calculus and stochastic processes. The current paper is a natural continuation of our previous paper \cite{CPS1}, devoted to hyperbolic SPDEs, also building onto this synergy of powerful tools. We then adopt here the same notation employed in
\cite{CPS1}, and a similar functional setting. We also mention that, recently, a white noise analysis of singular SPDEs has been performed in \cite{GMN22}, employing
Watanabe Sobolev spaces, which differs by the weighted Sobolev spaces we used in \cite{CPS1} and use again here.

Henceforth, in this paper we will  present techniques for solving stochastic partial differential equations of Schr\"odinger type resulting from the integration of these, nowadays classical, two powerful tools: chaos expansions and pseudodifferential techniques.
The model on which we will focus is an initial value (that is, Cauchy) problem for a differential operator of Schr\"odinger type on a curved space, which we will study globally on $\R^d$, namely,
\begin{equation}\label{eq:SPDE}
\left\{
\begin{array}{l} 
\bL(x,\partial_t,\partial_x;\omega)\lozenge u(t,x;\omega)= -i\partial_t u(t,x;\omega) + \bP(x,\partial_x;\omega)\lozenge u(t,x;\omega)= \bF(t,x,u(t,x;\omega)), \quad (t,x)\in[0,T]\times\R^d,\;\omega\in\Omega,
\\
u(0,x;\omega)=u_0(x;\omega),\quad x\in\R^d,\;\omega\in\Omega,
\end{array}\right.
\end{equation}
where $(\Omega, \mathcal{F}, P)$ is a probability space, $\lozenge$ denotes the Wick product (whose definition is recalled in Section \ref{subs:wpdsoce}), while $\bP$ plays the role of the stochastic Hamiltonian and $\bF$ introduces nonlinear perturbations into the equation (specific assumptions on these operators will be provided in Section \ref{sec:globsol}).
Note that the action of $\bL$ and $\bP$ by $\lozenge$ in \eqref{eq:SPDE} is a shorthand notation, since, for instance, the differential parts act as such, as it will be precisely described in Section \ref{sec:globsol} below. Explicitly:
\begin{itemize} 
\item $\bP$ is a stochastic analog (and a generalization) of a partial differential operator of the form
\[
H=\frac12\sum_{j,\ell=1}^d\partial_{x_j}\left(a_{j\ell}(x)\partial_{x_\ell}\right)+\sum_{j=1}^d m_{1j}(x)\partial_{x_j}+V(x;\omega),
\]
allowing for randomness in the potential term $V$, while the magnetic terms $m_{1j}$ and the geometry of the space, encoded into the coefficients $a_{j\ell}$ 
(see Remark \ref{rem:coeff}), are kept deterministic (see Section \ref{sec:globsol} below for the general form and the precise hypotheses); 
\item $\bF$, the diffusion term, is a real-valued function, subject to certain regularity conditions
(see below);
\item $u$ is an unknown stochastic process, called \emph{solution} of the Cauchy problem \eqref{eq:SPDE}.
\end{itemize}

We will employ chaos expansions, in connection with the properties of the solution operator of the associated
deterministic Schr\"odinger operator, defined through objects globally defined on $\R^d$, similarly to our
analysis of the hyperbolic Cauchy problems in this setting.
The main idea we use in this paper relies on the chaos expansion method: first, one uses the chaos expansion
of all stochastic data in the equation to convert the SPDE into an infinite system of deterministic PDEs, then the PDEs are recursively solved, and finally one must sum up these solutions to obtain the chaos expansion form of the solution of the initial SPDE. The
crucial point is to prove convergence of the series given by the
chaos expansion that defines the solution, and this part relies on obtaining good energy estimates of the PDE solutions, proving their regularity and using estimates on the Wick products. This approach has many advantages. Most notably, it provides an \textit{explicit form of the solution} of the SPDE, from which one
can directly compute the expectation, variance and other moments. It is
convenient also for numerical approximations, by truncating the series
in the chaos expansion to finite sums. Elements of these techniques and the corresponding notation are recalled in Appendix \ref{sec:prel}.

The second main tool we use in this paper is the $SG$ calculus of pseudodifferential operators (further abbreviated as $SG$ theory).
For the convenience of the reader, a short summary of the notation and the main features of
the $SG$ calculus are given in Appendix \ref{subs:sgcalc}. In particular, we will rely on results about Schr\"odinger type operators due to Craig \cite{craig}.

\medskip

The paper is organized as follows.  Section \ref{sec:globsol} is devoted to proving the first main result of the paper, that is, 
existence and uniqueness of a local in time solution to the linear version of equation \eqref{eq:SPDE}.  In the subsequent Section \ref{sec:globsolsemilin}, we prove our second main result, namely, 
existence and uniqueness of a local in time solution to the semilinear equation \eqref{eq:SPDE}.
In Section \ref{sec:wickp} we prove our third main result, namely, existence and uniqueness of a local in time solution to the 
nonlinear equation where the diffusion term takes on the form of Wick-powers, specifically, Wick-squares in equation \eqref{eq:SPDE}.
In the Appendix we have included a short summary of basic results about the two main tools we employ:
in Appendix \ref{sec:prel}, we provide the notation and an overview of the white noise analysis theory, including chaos expansions of generalized stochastic processes, Wick products and stochastic differential operators; in Appendix \ref{subs:sgcalc},
we recall the notation and fundamental notions of the $SG$ pseudodifferential calculus and the associated weighted Sobolev spaces.

\section*{Acknowledgements}
The first author has been partially supported by his own INdAM-GNAMPA Project 2023, Grant Code CUP$\_$E53C22001930001, and by
the Italian Ministry of the University and Research - MUR, within the framework of the Call relating to the scrolling of the final rankings of the 
PRIN 2022 - Project Code 2022HCLAZ8, CUP D53C24003370006 (PI A. Palmieri, Local unit Sc. Resp. S. Coriasco).
The first author also gratefully acknowledges the support by the Department of Mathematics and Informatics of the University of Novi Sad (Serbia), during his stays there in A.Y. 2022/2023 and A.Y. 2023/2024, when most of the results illustrated in this paper have been obtained. The second author was supported by the Serbian Academy of Sciences and Arts, project F10. The third author gratefully acknowledges the financial support of the Ministry of Science, Technological Development and Innovation of the Republic of Serbia (Grants No. 451-03-66\slash2024-03\slash200125 and 451-03-65\slash2024-03\slash200125). The authors are grateful to the anonymous Referee, for the careful reading of the 
manuscript, the constructive criticism and the suggestions, aimed at improving the overall quality of the paper.

%
%
%

\section{Solutions of linear magnetic stochastic Schr\"odinger equations on $\R^d$}\label{sec:globsol}
In this section we treat the Cauchy problems \eqref{eq:SPDE}, associated with a linear magnetic Schr\"odinger operators of the form
\beqs\label{elle}
\bL=-i\partial_t+\bP,
\eeqs
with coefficients globally defined and polynomially bounded on the whole Euclidean space $\R^d$, as will be in detail described in Assumptions \ref{hyp:P}. We refer the reader to \cite{cordes, CPS1}, Appendix \ref{sec:prel} and Appendix \ref{subs:sgcalc}, 
for notation, definition of the symbol classes $S^{m,\mu}$, the associated operators, and the properties of the scale of (Sobolev-Kato type) spaces, 
on which such operators naturally act. In particular, we need to introduce a subclass of the Sobolev-Kato spaces, of which we recall here below the
definition.
\begin{definition}\label{def:Hpint}
	\begin{itemize}
	\item[(i)] For any $(s,\sigma)\in\R^2$, the Sobolev-Kato space is defined as
	\begin{equation}\label{eq:defHssigmaalt}
		H^{s,\sigma}(\R^d)=\{u\in\caS^\prime(\R^d)\colon \jap^s u\in H^\sigma(\R^d)\},
	\end{equation}
	where $H^\sigma(\R^d)$ is the usual Sobolev space of order $\sigma$ on $\R^d$ and $\<y\>^s=(1+|y|^2)^\frac{s}{2}$, $y\in\R^d$.
	\item[(ii)] For any $z\in\N$, $\zeta\in\R$, define $\HzzetaRd :=\ds\bigcap_{j=0}^z H^{z-j,j+\zeta}(\R^d)$.
		The spaces $\mathcal H_{z,\zeta}(\R^d)$ are equipped with the norm 
		\begin{equation}\label{eq:normHpint}
			\|u\|_{\HzzetaRd}:=\ds\sum_{j=0}^z \|u\|_{H^{z-j,j+\zeta}(\R^d)}.
		\end{equation}
	\end{itemize}
\end{definition}
\noindent
By the properties of the Sobolev-Kato spaces recalled in Appendix \ref{subs:sgcalc}, it follows that 
$H^{z,z+\zeta}(\R^d)\subset \mathcal H_{z,\zeta}(\R^d)\subset H^{z,\zeta}(\R^d)$.
\begin{remark}\label{rem:Hpint}
\begin{itemize}
\item[(i)] Recall that the spaces $H^{r,\rho}$ with $r\geq 0$ and $\rho>d/2$ are algebras. This implies that also
the space $\Hzzeta$ is an algebra for $\zeta>d/2$. 
\item[(ii)] The spaces based on the norm \eqref{eq:normHpint} for an arbitrary $\zeta\in\N$ appear in \cite[Page XX-12]{craig}, where, in particular,
	the \textit{unweighted} Sobolev spaces $H^{0,\rho}$ are denoted by $H^\rho$, and the spaces , here
	$\caH_{r,0}$, of spatial moments up to order $r\in\N$, are denoted by $W^r$.
\end{itemize}
%
\end{remark}

The operator $$\bP(x,D_x;\omega): C([0,T],\Hzzeta(\R^d))\otimes(S)_{-1} \rightarrow C([0,T],\Hzzeta(\R^d))\otimes(S)_{-1}$$ is a stochastic operator in the sense of Lemma \ref{doralemma2}, acting as a spatial differential operator and stochastic (Wick) multiplication operator. It consists of a family of deterministic operators $P_\alpha = P_\alpha(x,D_x)$, $\alpha\in\mathcal I$, each mapping $C([0,T],\Hzzeta(\R^d))$ into itself.

\medskip

Recall, $\bP$ acts onto $u=u(t,x;\omega)=\sum_{\gamma\in\caI}u_\gamma(t,x) H_\gamma(\omega) \in C([0,T],\Hzzeta(\R^d))\otimes(S)_{-1} $ as

\begin{equation}\label{eq:Paction}
	(\bP\lozenge u)(t,x;\omega)=
	\sum_{\gamma\in\caI}\left[\sum_{\beta+\lambda=\gamma}(P_\beta u_\gamma)(t,x)\right]\cdot H_\gamma(\omega).
\end{equation}

\medskip
Now we list some assumptions that will make the operator $\bP$ be well-defined, and incorporate sufficient conditions that will ensure the solvability, in our chosen stochastic setting, of the equation 

\[
	\bL\lozenge u=-i\partial_t u+\bP\lozenge u =0.
\]

\begin{assumption}\label{hyp:P}
	Let $\bP$ be such that:
\begin{itemize}
	\item its expectation, that is, principal part, is of the form:
\begin{equation}\label{eq:operatortop}
\begin{aligned}
 P_{\minull}=P(x,\partial_x) &=\ds\frac12\sum_{j,\ell=1}^d\partial_{x_j}\left(a_{j\ell}(x)\partial_{x_\ell}\right)+m_1(x,-i\partial_x)+m_{0,\minull}(x,-i\partial_x)
 \\
 &=a(x,D_x) +a_1(x,D_x) +m_1(x,D_x)+m_{0,\minull}(x,D_x),
 \end{aligned}
\end{equation}
having set, as usual, $D_x=-i\partial_x$; 
\item the symbols appearing in the \textit{principal part} $P_{\minull}$ of $\bP$, namely,
\\
$a(x,\xi):=-\ds\frac12\sum_{j,\ell=1}^d a_{j\ell}(x)\xi_j\xi_\ell, \quad a_{j\ell}=a_{\ell j}$, $j,\ell=1,\dots,d$, \textit{Hamiltonian} of the equation, \\
$a_1(x,\xi) := \ds\frac i2 \ds\sum_{j,\ell=1}^d \partial_{x_j}a_{j\ell}(x)\xi_\ell$,\\ $m_1(x,\xi)$ coming from the magnetic field, and $m_{0,\minull}(x,\xi)$ the expectation of the potential term,\\
are such that (see \cite{craig}):
\begin{enumerate}
\item\label{ctr:aorderzero} the \textit{Hamiltonian} satisfies $a\in S^{0,2}(\R^d)$;
\item \label{ctr:aorder1}the \textit{lower order metric terms} satisfy $a_1\in S^{-1,1}(\R^d)$;
\item \label{ctr:ell} $a$ satisfies, for all $x,\xi\in\R^d$, $C^{-1}|\xi|^2\leq a(x,\xi)\leq C|\xi|^2$;
\item the \textit{magnetic field term} satisfies $m_1\in S^{0,1}(\R^d)$ and is real-valued;
\item the expected value of the \textit{potential} satisfies $m_{0,\minull}\in S^{0,0}(\R^d)$;
\end{enumerate}
\item the non-principal parts of the operator $P_{\beta}= P_\beta(x,\partial_x)=m_{0\beta}(x,D_x)$, $\beta\in\caI$, $\beta\not=\minull$, are such that:
\begin{enumerate}\setcounter{enumi}{5}
	\item $m_{0\beta}$
	satisfies $m_{0\beta}\in S^{0,0}(\R^d)$, $\beta\in\caI$, $\beta\not={\minull}$; 
	\item there exists $r\geq 0$ such that 	
\begin{equation}\label{eq:LKest}
	\sum_{\beta\in\mathcal I\atop\beta\neq\minull} 
		\|P_{\beta}\|_{\caL(\, C([0,T],\Hzzeta(\R^d)),
			\,
			C([0,T],\Hzzeta(\R^d)
			\,)} 
		(2\mathbb N)^{-\frac{r}{2}\beta}<\infty.
\end{equation}
\end{enumerate}
\end{itemize}
\end{assumption}

\begin{remark}
	In the deterministic case, a basic model of magnetic Schr\"odinger operator is 
	\[
		Q=\frac{1}{2}\left[\sum_{j,k=1}^d Q_j g_{j\ell}(x) Q_\ell-V(x)\right], \quad Q_j= hD_j-\mu A_j(x),
	\]
	with $h>0$ a (small) \textit{Plank constant} and $\mu>0$ a (large) coupling constant. The functions $g_{j\ell},A_j,V$, $j,k=1,\dots,d$, are usually
	assumed to be smooth and real-valued. The coefficients $g_{jl(x)}$ encode the curved geometry of the space, the functions $(A_1(x),\cdots,A_d(x))$ relate to the electromagnetic vector potential, while $V(x)$ is the scalar potential of the electric field.
	
	For physical reasons, it is natural to assume that $V$ might be random (underlying some fluctuations and uncertainty), but keeping the geometry of the space and the magnetic potential deterministic. Hence, we assume that $V$ is a spatial stochastic process with expansion $V(x;\omega)=\sum_{\alpha\in\caI}V_\alpha(x)H_\alpha(\omega)$.

	It is straightforward to check that the stochastic counterpart of this operator will have the form 
	\[
		Q=a(x,D_x)+a_1(x,D_x)+m_1(x,D_x)+m_0(x,D_x;\omega),
	\]
	where
	\[
		a_{j\ell}=-h^2g_{j\ell}, \quad m_1(x,\xi)=-\frac{h\mu}{2}\sum_{j,\ell=1}^d A_j(x)g_{j\ell}(x)\xi_\ell, 
		\quad m_0(x,\xi;\omega)=\frac{1}{2}\left[\mu^2\sum_{j,\ell=1}^d g_{j\ell}(x) A_j(x)\,A_\ell(x)-V(x;\omega)\right],
	\]
	hence it is clear that
	$$E(Q)=P_{\minull}$$
with $a,a_1,m_1,m_{0,\minull}$ as in \eqref{eq:operatortop}, and $$m_{0,\minull}(x,\xi)=E(m_0(x,\xi;\omega))=\frac{1}{2}\left[\mu^2\sum_{j,\ell=1}^d g_{j\ell}(x) A_j(x)\,A_\ell(x)-V_{\minull}(x)\right].$$
\end{remark}
We first recall key results in the analysis of the deterministic Schr\"odinger operators of the type we are considering, proved in
\cite{craig} (see also, e.g., \cite{CW,debouard,KPV05,KPV07,yajima}). 
\begin{theorem}[{\cite[Page XX-12]{craig}}]\label{T14craig}
Under Assumptions \ref{hyp:P}, the solution $u(t)$ to the associated deterministic Cauchy problem \eqref{eq:SPDE} with $u_0\in\HzzetaRd$, $\bF\equiv 0$ and $P_\gamma\equiv0$, 
$\gamma\not=\minull$, satisfies the estimate
\[\|u(t)\|_{\mathcal H_{z,\zeta}(\R^d)}\leq e^{C_{z,\zeta}t}\|u_0\|_{\mathcal H_{z,\zeta}(\R^d)},\quad t\in[0,T_0], \]
for $T_0\in(0,T]$ and a positive constant $C_{z,\zeta}$ depending only on $z,\zeta\in\N$.
\end{theorem}
\begin{remark}\label{rem:coeff}
	\begin{enumerate}
	\item [(i)] The symbol spaces $S^{m,\mu}$ are denoted by $S^{\mu,m}(1,0)$ in \cite{craig}, where it is remarked that the 
	ellipticity condition \eqref{ctr:ell}, together with the other hypotheses on $a$ and $a_1$,
	implies that the matrix $(a_{j\ell})$ is invertible, as well as that the Riemannian metric given by 
	the matrix $(a_{j\ell})^{-1}=(a^{j\ell})=\mathfrak{a}$ is asymptotically flat.
	\item[(ii)] By the hypotheses on $\mathfrak{a}$, our analysis actually covers the case
	\[
		P(x,\partial_x)=\frac{1}{2}\Delta_a+\widetilde{m}_1(x,\partial_x)+m_0(x,\partial_x),
	\]
	where $\widetilde m_1\in S^{0,1}$, and $\Delta_a$ is the Laplace-Beltrami operator associated with $\mathfrak{a}$, see \cite[p.XX-4]{craig}.
	\end{enumerate}
\end{remark}
\begin{remark}\label{rem:S}
As a consequence of Theorem \ref{T14craig}, the \textit{propagator} $S$ (or, equivalently, the fundamental solution) of $P$ defines continuous maps $S(t):\Hzzeta\rightarrow \Hzzeta$, whose norms can be bounded by $e^{C_{z,\zeta}t}$, $t\in[0,T_0]$, $z,\zeta\in\N$.
\end{remark}
We can now prove the first main result of the paper, which is the next Theorem \ref{thm:mainglob}.
\begin{theorem}\label{thm:mainglob}
	Let $\bP$ in \eqref{eq:SPDE} satisfy Assumptions \ref{hyp:P}. 
	%
	Assume also $u_0\in \Hzzeta(\R^d)\otimes (S)_{-1,-r}$ and $\bF\equiv0$.
	Then, there exists a time-horizon $T^\prime\in(0,T]$ such that the 
	homogeneous linear Cauchy problem \eqref{eq:SPDE} 
	admits a unique solution
	 $\displaystyle u\in  C([0,T^\prime], \Hzzeta(\R^d))\otimes (S)_{-1,-r}$.
\end{theorem}
\begin{proof}
	Employing \eqref{eq:Paction}, and
	writing $\displaystyle u_0=\sum_{\gamma\in\caI}u_{0\gamma}H_\gamma$, $u_{0\gamma}\in\Hzzeta$,
	we obtain an infinite dimensional system equivalent to \eqref{eq:SPDE}:
\begin{align*}
	[-i\partial_t + P_{\minull}]u_{\minull} & =  0,\qquad u_{\minull}(0)=u_{0,\minull}	, \mbox{ for } \gamma=\minull\\
	[-i\partial_t + P_{\minull}]u_{\gamma} & =  - \sum_{0\leq\lambda<\gamma}P_{\gamma-\lambda}u_\lambda,\qquad u_{\gamma}(0)=u_{0\gamma}	, \mbox{ for } \gamma\in\mathcal I\setminus\minull.
\end{align*}	
Their solutions are given by
	\begin{equation}\label{eq:syssol}
		u_\gamma(t) = S(t)u_{0\gamma}-
		i\int_0^t S(t-s)\left[\sum_{0\leq\lambda<\gamma}P_{\gamma-\lambda}u_\lambda(s)\right]ds, \quad t\in[0,T], \gamma\in\mathcal{I},
	\end{equation}
	where $S(t)$ depends only on $P_{\minull}=:P$ and has the property stated in Remark \ref{rem:S}.
	Notice that, by the regularity of the solutions 
	and the fact that all operators $P_\delta$ with $\delta\not=\minull$ 
	are in $\caO(0,0)$, Theorem \ref{thm:sobcont} implies that
	for each $\delta\in\mathcal I$, $\delta\not=\minull$, there exists a constant $K_\delta>0$ such that, for all $\lambda\in\mathcal I$,
	\[
		\|P_\delta u_\lambda(t)\|_{\Hzzeta} \leq  K_\delta \| u_\lambda(t)\|_{\Hzzeta}, 
		\quad t\in [0,T],
	\]
	By \eqref{eq:syssol}, with some other constant $C>0$, depending only on $P,z,\zeta,T,d$,
	\[
		\| u_\gamma\|_{C([0,T], \Hzzeta)} \le C\left(\|u_{0\gamma}\|_{\Hzzeta}+
		T\left\|\sum_{0\leq\lambda<\gamma}P_{\gamma-\lambda}u_\lambda\right\|_{C([0,T],\Hzzeta)}\right).
	\]
	Thus, for a new constant $\wtC>0$,
	\[
	\sum_{\gamma\in\mathcal I}\| u_\gamma\|_{C([0,T], \Hzzeta)}^2 \,
	(2\mathbb N)^{-r\gamma}
	\leq  \wtC \sum_{\gamma\in\mathcal I}
	\left[\|u_{0\gamma}\|^2_{\Hzzeta}
	+T^2 \left(
	\sum_{0\leq\lambda<\gamma}
	K_{\gamma-\lambda}\|u_\lambda\|_{C([0,T],\Hzzeta)}\right)^2
	\right](2\mathbb N)^{-r\gamma}.
	\]
	By the assumption $u_0\in\Hzzeta(\R^d)\otimes (S)_{-1,-r}$, we observe that
	\[
	M_I=\sum_{\gamma\in\mathcal I}\|u_{0\gamma}\|^2_{\Hzzeta}(2\mathbb N)^{-r\gamma}<\infty.
	\]
	Moreover, by immediate estimates, we obtain
	\begin{align*}
	\sum_{\gamma\in\mathcal I}&\left(
	\sum_{0\leq\lambda<\gamma}
	K_{\gamma-\lambda}\|u_\lambda\|_{C([0,T],\Hzzeta)}\right)^2(2\mathbb N)^{-r\gamma}
	=\sum_{\gamma\in\mathcal I}\left(
	\sum_{0\leq\lambda<\gamma}
	K_{\gamma-\lambda}(2\mathbb N)^{-\frac{r(\gamma-\lambda)}{2}}
	\|u_\lambda\|_{C([0,T],\Hzzeta)}(2\mathbb N)^{-\frac{r\lambda}{2}}\right)^2
	\\
	&\leq \left[\sum_{\delta\in\mathcal I\atop\delta\neq\minull}K_{\delta} (2\mathbb N)^{-\frac{r}{2}\delta}\right]^2
	\sum_{\gamma\in\mathcal I}\|u_\gamma\|^2_{C([0,T],\Hzzeta)}
	(2\mathbb N)^{-r\gamma}
	 \le M_L^2\sum_{\gamma\in\mathcal I}\|u_\gamma\|^2_{C([0,T],\Hzzeta)}
	(2\mathbb N)^{-r\gamma},
	\end{align*}
	where, by \eqref{eq:LKest},
	\[
		M_L= \sum_{\delta\in\mathcal I\atop\delta\neq\minull} K_{\delta} (2\mathbb N)^{-\frac{r}{2}\delta}<\infty.
	\]
	Then, after reducing $T$ to $T^\prime \in(0,T]$, we see that 
	\[
		\|u\|^2_{C([0,T^\prime],\Hzzeta(\R^d))\otimes (S)_{-1,-r}}
		\leq \frac{\wtC M_I}{1-\wtC(M_L T^\prime)^2} \,.
	\]
	The proof is complete.
\end{proof}
We observe that the solution exhibits the unbiasedness property, that is, its expectation coincides with the solution of the associated PDE 
obtained by taking expectations of all stochastic elements in \eqref{eq:SPDE}.

%

\section{Solutions of semilinear magnetic stochastic Schr\"odinger equations on $\R^d$}\label{sec:globsolsemilin}
We first introduce a class of maps on the solution spaces of \eqref{eq:SPDE}, similar to those appearing in \cite{ACS19b}.
\begin{definition}\label{def:lip} We say that a function 
$g:[0,T]\times\R^d\times(\HzzetaRd\otimes (S)_{-1,-r})\longrightarrow\HzzetaRd\otimes (S)_{-1,-r}$ belongs to the space $\Lip(z,\zeta)$, for chosen 
$z\in\N$, $\zeta\in[0,+\infty)$, if there exists a real valued and non-negative function $C_t=C(t)\in C([0,T])$ such that:
\begin{itemize}
\item for any $v\in \HzzetaRd\otimes (S)_{-1,-r}$, $t\in[0,T]$, we have 
\[
  \|g(t,\cdot,v)\|_{\HzzetaRd\otimes (S)_{-1,-r}}\leq C(t)\left[1+\|v\|_{\HzzetaRd\otimes (S)_{-1,-r}}\right];
\]
\item for any $v_1,v_2\in \HzzetaRd\otimes (S)_{-1,-r} $, $t\in[0,T]$, we have 
\[
	\|g(t,\cdot,v_1)-g(t,\cdot,v_2)\|_{\HzzetaRd\otimes (S)_{-1,-r}}\leq C(t)\|v_1-v_2\|_{\HzzetaRd\otimes (S)_{-1,-r}}.
\]
\end{itemize}
%
If the properties above are true only for $v,v_1,v_2\in U$,  with $U$ an open subset of $\HzzetaRd\otimes (S)_{-1,-r}$, then we say that $g\in\Liploc(z,\zeta)$. 
\end{definition}
\begin{remark}
	\begin{enumerate}
		\item[(i)] In applications, the open subset $U$ in Definition \ref{def:lip} 
			is usually a suitably small neighbourhood of $u_0$ in \eqref{eq:SPDE}.	
		\item[(ii)] Recall that $\Hzzeta$ is an algebra for $z,\zeta\in\N$, $\zeta>d/2$, since this is true for 
			$H^{s,\sigma}$, $s\ge0$, $\sigma>d/2$, and so is, obviously, $C([0,T], \Hzzeta)$. However, 
			this does not hold true for the solution space $C([0,T], \Hzzeta)\otimes(S)_{-1,-r}$. The reason for this is that the Wick product of two elements does not stay on the same level, e.g. if $F,G\in (S)_{-1,-p}$ then $F\lozenge G\in(S)_{-1,-2p-2}$, see \cite{HOUZ}. So, while $(S)_{-1}$ is an algebra, unfortunately $(S)_{-1,-p}$ for fixed $p$ is not, and the fixed point iteration needs a mapping of a Hilbert space into itself.		
			 Then, to treat
			nonlinearities of type $u^{\lozenge n}$ we will need a different approach, see Section \ref{sec:wickp} below.
	\end{enumerate}
\end{remark}
\begin{remark} 
Some operators that are of Lipschitz class in sense of Definition \ref{def:lip} would be coordinatewise stochastic operators,
that is, operators $G:\HzzetaRd\otimes(S)_{-1,-r}\rightarrow \HzzetaRd\otimes(S)_{-1,-r}$ that are composed of a family of deterministic operators $G_\alpha$, $\alpha\in\caI$, each one of Lipschitz class (either uniformly Lipschitz or their Lipschitz constants $L_\alpha$ satisfying certain growth rate), acting in the following manner:
$$G(u)=G(\sum_{\alpha\in\caI}u_\alpha H_\alpha) = \sum_{\alpha\in\caI}G_\alpha(u_\alpha) H_\alpha.$$ 
Indeed, for $v_1,v_2\in \HzzetaRd\otimes(S)_{-1,-r}$ we have
$$\| G(v_1)-G(v_2)\|^2_{\HzzetaRd\otimes(S)_{-1,-r}}\leq \sum_{\alpha\in\caI} \|G_\alpha(v_{1\alpha})-G_\alpha(v_{2\alpha})\|^2_{\HzzetaRd}(2\mathbb N)^{-r\alpha}\leq \sum_{\alpha\in\caI} L_\alpha^2
\|v_{1\alpha}-v_{2\alpha}\|^2_{\HzzetaRd}(2\mathbb N)^{-r\alpha}.$$
Now, if there is $L>0$ such that $L_\alpha\leq L$, $\alpha\in\caI$, or if $L:=\sum_{\alpha\in\caI} L_\alpha^2<\infty$, then one can easily obtain that
$$\| G(v_1)-G(v_2)\|^2_{\HzzetaRd\otimes(S)_{-1,-r}}\leq L \sum_{\alpha\in\caI}  \|v_{1\alpha}-v_{2\alpha}\|^2_{\HzzetaRd}(2\mathbb N)^{-r\alpha} = L \| v_1-v_2\|^2_{\HzzetaRd\otimes(S)_{-1,-r}}.$$
\end{remark}	
\begin{assumption}\label{hyp:F}
	Let $\bF$ in the right-hand side of \eqref{eq:SPDE} satisfy $\bF\in\Liploc(z,\zeta)$ on an open subset $U\subseteq \Hzzeta(\R^d)\otimes (S)_{-1,-r}$, for fixed $z,\zeta\in\N$, and $r\geq0$.
\end{assumption}
\begin{theorem} For fixed $z,\zeta\in\N$, let $\bP$ and $\bF$ in \eqref{eq:SPDE} satisfy Assumptions \ref{hyp:P} and \ref{hyp:F}, respectively. 
	%
	Assume also $u_0\in U$. 
	Then, there exists a time-horizon $T^\prime\in(0,T]$ such that \eqref{eq:SPDE} admits a unique solution 
	in $C([0,T^\prime],\HzzetaRd)\otimes(S)_{-1,-r}$. 
\end{theorem}
\begin{proof}
Notice that, in Theorem \ref{thm:mainglob}, we have proved the existence
of a fundamental solution operator for $\bL$, namely, $\frS(t)\colon \Hzzeta\otimes (S)_{-1,-r} \to \Hzzeta\otimes (S)_{-1,-r}\colon u_0\mapsto u(t)=\frS(t)u_0$, $u(t)$ the solution of \eqref{eq:SPDE} with initial datum $u_0$ and $\bF\equiv0$, $t\in[0,T^\prime]$. Notice also that, by the argument in the proof of Theorem \ref{thm:mainglob}, it also follows that $\frS$ is a continuous, uniformly bounded family of operators in $\caL(\Hzzeta\otimes (S)_{-1,-r}, \Hzzeta\otimes (S)_{-1,-r})$, such that $\frS(0)=I$, the identity operator. Then, the semilinear version of \eqref{eq:SPDE} is equivalent to the integral equation
\begin{equation}\label{eq:SPDEint}
	u(t)=\frS(t)u_0+\int_0^t \frS(t-s)\,\bF(s,\cdot,u(s))\,ds.
\end{equation}
We will show that, by the continuity of $\frS$ and the hypotheses, possibily after further reducing $T^\prime\in(0,T]$,
the right-hand side of \eqref{eq:SPDEint} is a strict contraction from $C([0,T^\prime], \Hzzeta)\otimes (S)_{-1,-r}$ to itself, which will prove the claim.
Indeed, let, for $u\in C([0,T^\prime], \Hzzeta)\otimes (S)_{-1,-r}$,
\[
	(\caT u)(t)= \frS(t)u_0+\int_0^t \frS(t-s)\,\bF(s,\cdot,u(s))\,ds.
\]
Then, by the hypotheses on $\bF$, setting $M_\frS= \max_{t\in[0,T^\prime]}\|\frS(t)\|_{\caL(\Hzzeta\otimes (S)_{-1,-r}, \Hzzeta\otimes (S)_{-1,-r})}$, $M_C= \max_{t\in[0,T]}C(t)$, we see that:
\begin{enumerate}
	\item[(i)] for any $T^\prime\in(0,T]$, $u\in C([0,T^\prime], \Hzzeta)\otimes (S)_{-1,-r}$, we have
	$\caT u\in C([0,T^\prime], \Hzzeta)\otimes (S)_{-1,-r}$; indeed, 
	\begin{align*}
		\|\caT u\|_{C([0,T^\prime], \Hzzeta)\otimes (S)_{-1,-r}}
		&\le
		M_\frS\|u_0\|_{\Hzzeta\otimes (S)_{-1,-r}}
		\\
		&+ 
		M_\frS\left[1+\|u\|_{C([0,T^\prime],\Hzzeta)\otimes (S)_{-1,-r})}\right]
		\int_0^{T^\prime} C(s)\,ds <+\infty;
	\end{align*}
	\item[(ii)] there exists $T^\prime\in(0,T]$ such that, for any $t\in[0,T^\prime]$, $u(t)\in U\Rightarrow \caT u(t)\in U$; in fact, there exists $\rho>0$ such that
	$\|v-u_0\|<\rho\Rightarrow v\in U$ and, for a suitable $T^\prime\in(0,T]$, for any $t\in[0,T^\prime]$, $u(t)\in U$,
	\begin{align*}
		\|\caT u(t)-u_0\|_{\Hzzeta\otimes (S)_{-1,-r}}&\le \|[\frS(t)-I]u_0\|_{\Hzzeta\otimes (S)_{-1,-r}} + M_\frS\int_0^{T^\prime}C(s)\left[1+\|u(s)\|_{\Hzzeta\otimes (S)_{-1,-r}}\right]ds
		\\
		&\le \|\frS(t)-I\|_{\caL(\Hzzeta\otimes (S)_{-1,-r}, \Hzzeta\otimes (S)_{-1,-r})} \|u_0\|_{\Hzzeta\otimes (S)_{-1,-r}}
		\\
		&+ M_\frS \,M_C\,(1+\rho+\|u_0\|_{\Hzzeta\otimes (S)_{-1,-r}})T^\prime
		\\
		&<\rho,
	\end{align*}
	by the continuity of $\frS(t)$ and $\frS(0)=I$, choosing $T^\prime\in(0,T]$ small enough;
	\item[(iii)] there exists $L>0$ such that, for any $u,v\in C([0,T^\prime], \Hzzeta\otimes (S)_{-1,-r})$, $u(t),v(t)\in U$, $t\in[0,T^\prime]$,
	\[
		\|\caT u - \caT v\|_{C([0,T^\prime], \Hzzeta\otimes (S)_{-1,-r}}\le 
		(LT^\prime)\|u-v\|_{C([0,T^\prime], \Hzzeta\otimes (S)_{-1,-r}};
	\]
	indeed,
	\begin{align*}
		\|(\caT u - \caT v)(t)\|_{\Hzzeta\otimes (S)_{-1,-r}}
		&\le
		M_\frS
		\int_0^{T^\prime}
		\| 
		\bF(s,\cdot, u(s))-\bF(s,\cdot,v(s))
		\|_{\Hzzeta\otimes (S)_{-1,-r}}
		\,ds
		\\
		&\le M_\frS
		\int_0^{T^\prime}C(s)
		\|u(s)-v(s)\|_{\Hzzeta\otimes (S)_{-1,-r}}\,ds
		\\
		&\le 
		M_\frS
		\|u-v\|_{C([0,T^\prime],\Hzzeta)\otimes (S)_{-1,-r})}
		\int_0^{T^\prime} C(s)\,ds
		\\
		&\Rightarrow
		\\
		\|\caT u - \caT v\|_{C([0,T^\prime],\Hzzeta)\otimes (S)_{-1,-r}}
		&\le
		(M_\frS\,M_C) \,T^\prime\,\|u-v\|_{C([0,T^\prime],\Hzzeta)\otimes (S)_{-1,-r})}.
	\end{align*}
\end{enumerate}
The proof is complete.
\end{proof}

%

\section{Wick-product nonlinearities}\label{sec:wickp}
Here we deal with the case of a diffusion term $\bF$ that is of non-Lipschitz type, but noteworthy and important from the physical point of view, namely, a power-nonlinearity of the form $\bF(u)=u^{\lozenge n}$, $n\in\mathbb N$. For technical simplicity we will fully elaborate only the case of  $n=2$, which is illustrative and already demands a fair piece of juggling with estimates related to Catalan numbers. 
Notice that the same procedure can be applied to higher order powers or even be adopted to polynomial nonlinearities (see \cite{Milica2}).
Beyond such Wick-type nonlinearities, one can explore nonlinearities in the form of Wick versions of analytic functions (see \cite{LPSZ23}).

Hence, the equation under consideration is now
\begin{equation}\label{wick2EQ}
-i\partial_t\;u +\bP\lozenge u +\lambda\;u^{\lozenge 2} = 0
\end{equation}
with suitable initial condition. Here, $\lambda>0$ refers to a repulsive nonlinearity, and $\lambda<0$ refers to an attractive nonlinearity, respectively.

\begin{remark}\label{rem:wickprphys}
The Wick product has received some criticisms about its physical feasibility (see, e.g., \cite{KL}), in particular, for not capturing the property of probabilistic independence.
However, it is closely related to the notion of renormalization in quantum physics, and represents the highest order approximation of the ordinary product (while some better 
approximations may be achieved in the framework of Malliavin derivatives). Hence, in cases of generalized stochastic processes, where the ordinary product is ill-defined, 
the Wick product represents a meaningful choice to model multiplication operators or other nonlinearities in the model equations (see, e.g., \cite{VWMRK13}).
\end{remark}

\medskip

Note that the chaos expansion representation of the Wick-square is given by
\begin{align}\label{WP2}
	u^{\lozenge 2}(t,x;\omega)&= \sum_{\alpha \in \mathcal I} \Big( \sum_{\gamma\leq \alpha } \, u_\gamma(t,x) \,\, u_{\alpha-\gamma} (t,x)\Big) \, H_\alpha(\omega)\\\nonumber
	&=u^2_{\mathbf{0}}(t,x)\,H_{\mathbf{0}}(\omega)+\sum_{|\alpha|>0} \Big(2u_{\mathbf{0}}(t,x)\,u_\alpha(t,x)+ \sum_{0<\gamma< \alpha } \, u_\gamma(t,x) \,\, u_{\alpha-\gamma} (t,x)\Big) \, H_\alpha(\omega),
\end{align}
where $t\in[0,T],x\in\mathbb R^d$, $\omega\in\Omega.$ For notational convenience below, from here on we denote $\mathbf{0}=\minull$.

\medskip

Equation \eqref{wick2EQ} is now equivalent to an infinite system of (deterministic Cauchy problems associated with evolution) PDEs, namely:
 	\begin{enumerate}
	\item[ i)] for $\alpha =\mathbf{0}$,
	\begin{equation}
		\label{nelinearna det-2}
		-i\partial_t u_{\mathbf{0}} (t,x) + P_{\mathbf{0}}(x,D_x) u_{\mathbf{0}}  (t,x) +  \lambda u^2_{\mathbf{0}} (t,x) = 0, \quad u_{\mathbf{0}}(0,x) = u_{\mathbf{0}}^0(x);
	\end{equation}
	\item[ii)] for  $\alpha >\mathbf{0}$,
	\begin{equation}
		\label{sistem 2-2}
		\Big(-i\partial_t +  P_{\mathbf{0}}(x,D_x) + 2\lambda u_{\mathbf{0}} (t,x) \Big) \, u_\alpha(t,x) + \sum_{\mathbf{0} < \gamma <  \alpha } \,P_\gamma(x) u_{\alpha-\gamma}(t,x)  + \lambda \sum_{\mathbf{0} < \gamma <  \alpha } \, u_\gamma(t,x) \,\, u_{\alpha-\gamma}(t,x) = 0, \quad  
		u_\alpha (0,x) =  u_\alpha^0 (x).
	\end{equation} 
\end{enumerate}
In all the equations \eqref{nelinearna det-2}-\eqref{sistem 2-2} of the system we have $t\in (0,T]$, $x\in\mathbb R^d$, $\omega\in\Omega$. The system \eqref{sistem 2-2} should be solved recursively on the length of $\alpha$. In each step, the solutions of the previous ones appear in the non-homogeneous part, while the operator is the same for each $\alpha >\mathbf{0}$.

\medskip

Note that in \eqref{sistem 2-2} we have a new operator (a perturbation of the original one by $u_{\mathbf{0}}$), that introduces a time-dependence into the potential term of the principal part. Let us denote this new operator as
\begin{equation}\label{opB}
	B(t,x, D_x) = P_{\mathbf{0}}(x,D_x) + 2\lambda u_{\mathbf{0}} (t,x),
\end{equation}
and let
$$g_\alpha(t,x) = \sum_{\mathbf{0} < \gamma <  \alpha } \,P_\gamma(x) u_{\alpha-\gamma}(t,x)  + \lambda \sum_{\mathbf{0} < \gamma <  \alpha } \, u_\gamma(t,x) \,\, u_{\alpha-\gamma}(t,x), \quad \alpha > \mathbf{0},$$
so that the system \eqref{sistem 2-2}  can be written in the form 
\begin{equation}
	\label{sistem 2a}
	-i\partial_t u_\alpha (t,x) +  B(t,x,D_x)  \, u_\alpha(t,x) +  g_{\alpha}(t,x) = 0, \quad  
	u_\alpha (0,x) =  u_\alpha^0(x), \quad \alpha> \mathbf{0} .  
\end{equation}

\begin{assumption}\label{hyp:Wick2} Assume that the following conditions hold:
	\begin{enumerate}
		\item the operator $\bP$ satisfies Assumption \ref{hyp:P} and, for fixed $z,\zeta\in\N$, there exists $r\geq 0$ such that 
		$\bP$ fulfills \eqref{eq:LKest};
		\item \label{hypA2} the initial value satisfies $u_0\in \Hzzeta(\R^d)\otimes (S)_{-1,-r}$;
		\item \label{hypA3} the deterministic nonlinear Cauchy problem \eqref{nelinearna det-2} with $u_{\mathbf{0}}^0=E(u_0)$ 
		has a classical solution $u_{\mathbf{0}}\in C([0,T],\Hzzeta)$. 
	\end{enumerate}
\end{assumption}

\begin{remark}\label{rem:B}
Note that, due to Assumptions \ref{hyp:Wick2},\eqref{hypA3}, and the fact that $\Hzzeta$ is an algebra, the new (time-perturbed) operator $B$ in \eqref{opB} will also generate an appropriate propagator system. Namely, as stated in Remark \ref{rem:S}, the operator $-i\partial_t+P_{\mathbf{0}}$ defines a stable family of infinitesimal generators $S(t)$ such that
$$\|S(t)\| \leq m e^{wt},\quad w=C_{z,\zeta}$$ holds. 
Denote
\begin{equation}\label{eq:M2}
	M_2= \sup_{t\in[0,T]}\|u_{\mathbf{0}}(t,x)\|_{\HzzetaRd}
\end{equation}
The perturbation is a multiplication operator, giving rise to a bounded linear operator $u_{\mathbf{0}}(t,x):\Hzzeta\rightarrow \Hzzeta$ such that
$$\|2\lambda u_{\mathbf{0}}(t,x)\cdot f(x)\|_{\Hzzeta} \leq 2|\lambda| \|u_{\mathbf{0}}(t,x)\|_{\Hzzeta} \| f(x)\|_{\Hzzeta} \leq 2|\lambda| M_2 \| f(x)\|_{\Hzzeta}.	
$$
Hence, $B(t,x,D_x)$ from \eqref{opB} will have a stable family of infinitesimal generators $\tilde S(t)$ such that
\begin{equation}\label{eq:w2}
	\|\tilde S(t)\|\leq m e^{(w+2|\lambda|M_2)t} = m e^{w_2t}, \mbox{ with } w_2=C_{z,\zeta}+2|\lambda|M_2,
\end{equation} 
holds for $t\in[0,T]$. The solution to each equation in \eqref{sistem 2a} will be given by
\begin{equation}\label{eq:resenjepert}
u_\alpha(t,x) = \tilde S(t) u_\alpha^0(x) -i \int_0^t \tilde S(t-s) g_\alpha(s,x)ds,\quad t\in[0,T].
\end{equation}
\end{remark}

\begin{remark}
Let $u_0\in \Hzzeta(\R^d)\otimes (S)_{-1,-r}$ be an initial condition satisfying Assumptions \ref{hyp:Wick2},\eqref{hypA2}. Then, there exists  $\tilde{K}>0$ such that $\sum_{\alpha\in\mathcal{I}}\|u^0_\alpha\|_{\Hzzeta}^2(2\mathbb{N})^{-\tilde{r}\alpha}=\tilde{K}$. There exists also $p\geq 0$ (possibly $p>>r$) and $K\in (0,1)$ such that $\sum_{\alpha\in \mathcal{I}}\|u^0_\alpha\|_{\Hzzeta}^2(2\mathbb{N})^{-2p\alpha}=K^2$, or, equivalently,
\begin{equation}\label{ocena p. uslova}
	\exists p\geq 0\;\exists K\in (0,1)\; \forall \alpha\in\mathcal{I} \quad \|u_\alpha^0\|_{\Hzzeta}\leq K(2\mathbb{N})^{p\alpha}.
\end{equation}

\medskip The same observation can be carried out to rewrite \eqref{eq:LKest}. Namely, there exist $P\in(0,1)$ and $q\geq0$ such that for all $\beta\in\caI\setminus\minull$ one has $\|P_\beta\|_{\caL(\, C([0,T],\Hzzeta(\R^d)),		\,
		C([0,T],\Hzzeta(\R^d)		\,)}\leq P (2\mathbb N)^{q\beta/2}.$
Without loss of generality (by taking maximums), we will assume that $K=P$ and $p=q/2$, hence
\begin{equation}\label{ocena op. uslova}
	\exists p\geq 0\;\exists K\in (0,1)\; \forall \beta\in\mathcal{I}\setminus \mathbf{0}
	\quad \|P_\beta\|_{\caL(\, C([0,T],\Hzzeta(\R^d)),		\,
		C([0,T],\Hzzeta(\R^d)		\,)} \leq K(2\mathbb{N})^{p\beta}.
\end{equation} 

\end{remark}

The next Theorem \ref{thm:main3} is the main result of this section.

\begin{theorem}\label{thm:main3}
Let Assumptions \ref{hyp:Wick2} be fulfilled. Then, there exists a unique solution 
$u\in C([0,T],\HzzetaRd)\otimes (S)_{-1}$ to the nonlinear stochastic equation \eqref{wick2EQ}.
\end{theorem}

\begin{proof}
According to Assumption \ref{hyp:Wick2},\eqref{hypA3} and Remark \ref{rem:B}, each equation in the system \eqref{nelinearna det-2}-\eqref{sistem 2-2} has a unique solution $u_\alpha(t,x)\in C([0,T],\Hzzeta)$, $\alpha\in\caI$, given by $u_{\mathbf{0}}$ in Assumptions \ref{hyp:Wick2},\eqref{hypA3}, and $u_\alpha$ in \eqref{eq:resenjepert} for $\alpha>\mathbf{0}$. Set
\begin{align*}
	L_\alpha:=\sup_{t\in[0,T]}\|u_\alpha(t)\|_{\Hzzeta},\quad \alpha\in\mathcal{I}.
\end{align*}

For $\alpha=\mathbf{0}$, using \eqref{eq:M2} we have
\begin{equation}\label{ocena L0-2}
	L_\mathbf{0}=\sup_{t\in[0,T]}\|u_\mathbf{0}(t)\|_{\Hzzeta}=M_2.
\end{equation}

Let $|\alpha|=1.$ Then $\alpha=\varepsilon_k,\;k\in \mathbb{N}$, and using \eqref{eq:resenjepert} we have that
$$\|u_{\varepsilon_k} (t)\|_{\Hzzeta}\leq \|\tilde S(t)\|\| u_{\varepsilon_k}^0\|_{\Hzzeta} + \int_0^t\|\tilde S (t-s)\|\|g_{\varepsilon_k}(s)\|_{\Hzzeta} ds, \quad t\in [0,T],$$
with $g_{\varepsilon_k}(s) = P_{\varepsilon_k}u_\mathbf{0}(s) = m_{0,\varepsilon_k}(x,D_x)u_\mathbf{0}(s,x)$, that can be estimated by \eqref{ocena op. uslova} in the following manner:
$$\sup_{s\in[0,t]}\|g_{\varepsilon_k}(s)\|_{\Hzzeta} \leq \|P_{\varepsilon_k}\|\sup_{s\in[0,t]}\|u_\mathbf{0}(s) \|_{\Hzzeta} \leq K (2\mathbb N)^{p\varepsilon_k}M_2.$$
From \eqref{eq:w2} we obtain 
\begin{align}\label{ocen int}
	\int_0^t\|\tilde S(t-s)\|ds\leq\int_0^t me^{w_2(t-s)}ds =m\frac{e^{w_2 t}-1}{w_2}\leq \frac{m}{w_2}e^{w_2 T},\quad t\in[0,T],\quad\alpha>\mathbf{0},
\end{align} 
and now \eqref{eq:w2}, \eqref{ocena p. uslova} and \eqref{ocena op. uslova} imply that
\begin{align}\label{ocena L_epsilon}
	L_{\varepsilon_k}&=\sup_{t\in[0,T]}\|u_{\varepsilon_k} (t)\|_{\Hzzeta}\leq\sup_{t\in[0,T]}\Big\{\|\tilde S (t)\|\| u_{\varepsilon_k}^0\|_{\Hzzeta}+\sup_{s\in[0,t]}\|g_{\varepsilon_k}(s)\|_{\Hzzeta}\int_0^t\|\tilde S(t-s)\|ds\Big\}\\\nonumber
	&\leq me^{w_2T}K(2\mathbb{N})^{p\varepsilon_k} +\frac{m}{w_2}e^{w_2 T}K(2\mathbb{N})^{p\varepsilon_k}M_2=m_1e^{w_2T}K(2\mathbb{N})^{p\varepsilon_k},\quad t\in[0,T],\quad k\in\mathbb{N},
\end{align}
where $m_1=m+\frac{m}{w_2}M_2.$

For $|\alpha|>1$ we consider two possibilities for $L_\alpha.$ First, if $L_\alpha\leq \sqrt{K}(2\mathbb{N})^{p\alpha}$ for all $|\alpha|>1$, then the statement of the theorem follows directly, since, for $q>2p+1$, keeping in mind \eqref{ocena L0-2} and \eqref{ocena L_epsilon}, we obtain
\begin{align*}
	\sum_{\alpha\in\mathcal{I}}\sup_{t\in[0,T]}\|u_\alpha(t)\|^2_{\Hzzeta}(2\mathbb{N}&)^{-q\alpha}=\sum_{\alpha\in\mathcal{I}} L_\alpha^2(2\mathbb{N})^{-q\alpha}=L_\mathbf{0}^2+\sum_{k\in \mathbb{N}}L_{\varepsilon_k}^2(2\mathbb{N})^{-q\varepsilon_k} + \sum_{|\alpha|>1}L_\alpha^2(2\mathbb{N})^{-q\alpha}\\
	&\leq M_2^2+(m_1e^{w_2 T}K)^2\sum_{k\in\mathbb{N}}(2\mathbb{N})^{(2p-q)\varepsilon_k} + K\sum_{|\alpha|>1}(2\mathbb{N})^{(2p-q)\alpha}<\infty,
\end{align*}
that is, $u\in  C([0,T],{\Hzzeta})\otimes (S)_{-1,-q}.$

The second case is if $L_\alpha> \sqrt{K}(2\mathbb{N})^{p\alpha}$ for some $\alpha\in \mathcal{I},\;|\alpha|>1.$ In what follows, we will assume the worst-case scenario that $L_\alpha> \sqrt{K}(2\mathbb{N})^{p\alpha}$ for all $\alpha\in \mathcal{I},\;|\alpha|>1$, and prove that even under that growth rate one can find $q>p$ large enough such that $\sum_{\alpha\in\mathcal{I}} L_\alpha^2(2\mathbb{N})^{-q\alpha}<\infty$ will follow at the end.

\medskip

Let $\alpha,\;|\alpha|>1$ be fixed. From \eqref{eq:resenjepert} we obtain

$$u_\alpha(t)=\tilde S(t)u^0_\alpha - i\int_0^t \tilde S(t-s)\Big[\lambda \sum_{0<\gamma<\alpha}u_{\alpha-\gamma}(s)u_\gamma(s) + \sum_{0<\gamma<\alpha}P_{\alpha-\gamma}u_\gamma(s)\Big]ds,\quad t\in [0,T].$$
From this we have
\begin{align*}
	L_\alpha &=\sup_{t\in[0,T]}\|u_\alpha(t)\|_{\Hzzeta} \\
	& \leq \sup_{t\in[0,T]}\Bigg\{\|\tilde S(t)\|\|u^0_\alpha\|_{\Hzzeta}+|\lambda|\int_0^t \|\tilde S(t-s)\|\Big\|\sum_{0<\gamma<\alpha}u_{\alpha-\gamma}(s)u_\gamma(s)\Big\|_{\Hzzeta}ds\\
	&+\int_0^t\|\tilde S(t-s)\|\Big\|\sum_{0<\gamma<\alpha}P_{\alpha-\gamma}u_\gamma(s)\Big\|_{\Hzzeta} ds\Bigg\}\\
	&\leq \sup_{t\in[0,T]}\Bigg\{me^{w_2t}\|u^0_\alpha\|_{\Hzzeta}+|\lambda|\sup_{s\in[0,t]}\sum_{0<\gamma<\alpha}\|u_{\alpha-\gamma}(s)\|_{\Hzzeta}\|u_\gamma(s)\|_{\Hzzeta}\cdot\int_0^t \|\tilde S(t-s)\|ds\\
	&+\sup_{s\in[0,t]}\sum_{0<\gamma<\alpha}\|P_{\alpha-\gamma}\|\|u_\gamma(s)\|_{\Hzzeta}\int_0^t\|\tilde S(t-s)\|ds\Bigg\}.
\end{align*}

Using \eqref{ocen int}, recalling \eqref{ocena p. uslova}-\eqref{ocena op. uslova}, we obtain
\begin{align*}
	L_\alpha &=\sup_{t\in[0,T]}\|u_\alpha(t)\|_{\Hzzeta}  \\
	& \leq me^{w_2T}\|u^0_\alpha\|_{\Hzzeta} +|\lambda|\frac{m}{w_2}e^{w_2T}\sum_{0<\gamma<\alpha}\sup_{t\in[0,T]}\|u_{\alpha-\gamma}(t)\|_{\Hzzeta}\sup_{t\in[0,T]}\|u_\gamma(t)\|_{\Hzzeta} \\
	&+\frac{m}{w_2}e^{w_2T}\sum_{0<\gamma<\alpha}K(2\mathbb N)^{p(\alpha-\gamma)}\sup_{s\in[0,T]}\|u_\gamma(s)\|_{\Hzzeta}\\
	& \leq me^{w_2T}K(2\mathbb{N})^{p\alpha}+|\lambda|\frac{m}{w_2}e^{w_2T}\sum_{\mathbf{0}<\gamma<\alpha}L_{\alpha-\gamma}L_\gamma + \frac{m}{w_2}e^{w_2T}\sum_{\mathbf{0}<\gamma<\alpha}K(2\mathbb N)^{p(\alpha-\gamma)}L_\gamma.
\end{align*} 
Now, since we assumed $L_\gamma > \sqrt{K}(2\mathbb N)^{p\gamma}$ for all $\gamma>0$, and since $K\in(0,1)$, it follows that 
$$\sum_{\mathbf{0}<\gamma<\alpha}K(2\mathbb N)^{p(\alpha-\gamma)}L_\gamma < \sum_{\mathbf{0}<\gamma<\alpha}L_{\alpha-\gamma}L_\gamma.$$ Hence,
$$L_\alpha \leq me^{w_2T}K(2\mathbb{N})^{p\alpha}+(|\lambda|+1)\frac{m}{w_2}e^{w_2T}\sum_{\mathbf{0}<\gamma<\alpha}L_{\alpha-\gamma}L_\gamma.$$
Let $m_2= \max\Big\{m, m_1, (|\lambda|+1)\frac{m}{w_2}\Big\}$. For this constant now we have
\begin{align}\label{ocena L_alpha}
	L_\alpha\leq m_2e^{w_2T}\Big(K(2\mathbb{N})^{p\alpha}+\sum_{\mathbf{0}<\gamma<\alpha}L_{\alpha-\gamma}L_\gamma\Big),\qquad \alpha>\mathbf{0},
\end{align}
and \eqref{ocena L_epsilon} holds as well, with $m_1$ replaced by $m_2$.

\medskip 

Let $\tilde{L}_\alpha,\;\alpha>\mathbf{0},$ be given by
\begin{align*}
	\tilde{L}_{\alpha}:=2m_2e^{w_2 T}\Big(\frac{L_\alpha}{\sqrt{K}(2\mathbb{N})^{p\alpha}}+1\Big).
\end{align*}
Thus, from \eqref{ocena L_epsilon} we have that for all $k\in \mathbb{N}$
\begin{align}\label{ocena tilda L_alpha}
	\tilde{L}_{\varepsilon_k}=2m_2e^{w_2 T}\Big(\frac{L_{\varepsilon_k}}{\sqrt{K}(2\mathbb{N})^{p\varepsilon_k}}+1\Big)&\leq 2m_2e^{w_2 T}\Big(\frac{m_2e^{w_2 T}K(2\mathbb{N})^{p\varepsilon_k}}{\sqrt{K}(2\mathbb{N})^{p\varepsilon_k}}+1\Big)\\\nonumber
	&=2m_2e^{w_2 T}(m_2e^{w_2 T}\sqrt{K}+1).
\end{align}
We proceed with the estimation of the term $\sum_{\mathbf{0}<\gamma<\alpha}\tilde L_\gamma\tilde L_{\alpha-\gamma}$ for given $|\alpha|>1$:
\begin{align*}
	\sum_{\mathbf{0}<\gamma<\alpha}\tilde L_\gamma\tilde L_{\alpha-\gamma}&=\sum_{\mathbf{0}<\gamma<\alpha} (2m_2e^{w_2 T})^2\Big(\frac{L_{\gamma}}{\sqrt{K}(2\mathbb{N})^{p\gamma}}+1\Big)\Big(\frac{L_{\alpha-\gamma}}{\sqrt{K}(2\mathbb{N})^{p(\alpha-\gamma)}}+1\Big)\\
	&\geq (2m_2e^{w_2 T})^2 \Big(\sum_{\mathbf{0}<\gamma<\alpha}\frac{L_\gamma L_{\alpha-\gamma}}{K(2\mathbb{N})^{p\alpha}}+1\Big)\\
	&=\frac{(2m_2e^{w_2 T})^2}{K(2\mathbb{N})^{p\alpha}}\sum_{\mathbf{0}<\gamma<\alpha}L_\gamma L_{\alpha-\gamma}+(2m_2e^{w_2 T})^2.
\end{align*}
Using the estimate \eqref{ocena L_alpha} we obtain
\begin{align*}
	\sum_{\mathbf{0}<\gamma<\alpha}\tilde L_\gamma\tilde L_{\alpha-\gamma}&\geq \frac{(2m_2e^{w_2 T})^2}{K(2\mathbb{N})^{p\alpha}}\Big(\frac{L_\alpha}{m_2e^{w_2 T}}-K(2\mathbb{N})^{p\alpha}\Big)+(2m_2e^{w_2 T})^2=\frac{4m_2e^{w_2 T}}{K(2\mathbb{N})^{p\alpha}}L_\alpha.
\end{align*}
Now, since $L_\alpha>\sqrt{K}(2\mathbb{N})^{p\alpha}$ for $\alpha>\mathbf{0}$, and since  $K<1$, we obtain
\begin{align*}
	\sum_{\mathbf{0}<\gamma<\alpha}\tilde L_\gamma\tilde L_{\alpha-\gamma}&\geq \frac{4m_2e^{w_2 T}}{\sqrt{K}(2\mathbb{N})^{p\alpha}}L_\alpha=\frac{2m_2e^{w_2 T}}{\sqrt{K}(2\mathbb{N})^{p\alpha}}L_\alpha+\frac{2m_2e^{w_2 T}}{\sqrt{K}(2\mathbb{N})^{p\alpha}}L_\alpha\\
	&\geq 2m_2e^{w_2 T}\Big(\frac{L_\alpha}{\sqrt{K}(2\mathbb{N})^{p\alpha}}+1\Big)=\tilde{L}_\alpha.
\end{align*}
Hence, for all $\alpha\in\caI$, $|\alpha|>1$, we have finally proved
\begin{align*}
	\sum_{\mathbf{0}<\gamma<\alpha}\tilde L_\gamma\tilde L_{\alpha-\gamma}\geq \tilde{L}_\alpha. 
\end{align*}

\medskip

Let $R_\alpha,\alpha>\mathbf{0},$ be defined as follows:
\begin{align*}
	R_{\varepsilon_k}&=\tilde{L}_{\varepsilon_k},\quad k\in\mathbb{N},\\
	R_\alpha&=\sum_{\mathbf{0}<\gamma<\alpha}R_\gamma R_{\alpha-\gamma},\quad |\alpha|>1.
\end{align*}
It is a direct consequence of the definition of the numbers $R_\alpha,\alpha>\mathbf{0},$ and it can be shown by induction with respect to the length of the multi-index $\alpha>\mathbf{0}$, that (see \cite[Section 5]{KL})
\begin{align}\label{L manje od R}
	\tilde{L}_\alpha\leq R_{\alpha},\quad \alpha>\mathbf{0}.
\end{align}
Lemma \ref{multi katalan} shows that the numbers $R_\alpha,\;\alpha>\mathbf{0}$, satisfy
\begin{align*}
	R_\alpha=\frac{1}{|\alpha|}\binom{2|\alpha|-2}{|\alpha|-1}\frac{|\alpha|!}{\alpha!}\prod_{i=1}^\infty R_{\varepsilon_i}^{\alpha_i},\quad \alpha>\mathbf{0}.
\end{align*}
By virtue of \eqref{ocena tilda L_alpha},
$$\prod_{i=1}^\infty R_{\varepsilon_i}^{\alpha_i}=\prod_{i=1}^\infty \tilde{L}_{\varepsilon_i}^{\alpha_i}\leq \prod_{i=1}^\infty (2m_2e^{w_2 T}(m_2e^{w_2 T}\sqrt{K}+1))^{\alpha_i}.$$
Let $c=2m_2e^{w_2 T}(m_2e^{w_2 T}\sqrt{K}+1).$ Then
\begin{align}\label{ocena R}
	R_\alpha\leq \mathbf{c}_{|\alpha|-1}\frac{|\alpha|!}{\alpha!}c^{|\alpha|},\quad \alpha >\mathbf{0},
\end{align}
where $\mathbf{c}_n=\frac{1}{n+1}\binom{2n}{n},\;n\geq 0$, denotes the $n$th Catalan number (more information on Catalan numbers is provided in Lemma \ref{katalan}). Using Lemma \ref{ocena faktorijela}, \eqref{L manje od R}, \eqref{ocena R} and \eqref{Cat ocena} we obtain that, for $\alpha\in\mathcal I$, $|\alpha|>1$, the estimation
\begin{align*}
	\tilde{L}_\alpha\leq R_\alpha \leq 4^{|\alpha|-1}(2\mathbb{N})^{2\alpha}c^{|\alpha|}
\end{align*} holds.
Finally, from the definition of $\tilde L_\alpha,\;\alpha>\mathbf{0}$, we obtain
\begin{align*}
	L_\alpha\leq \Big(\frac{4^{|\alpha|-1}(2\mathbb{N})^{2\alpha}c^{|\alpha|}}{2m_2e^{w_2 T}}-1\Big)\sqrt{K}(2\mathbb{N})^{p\alpha}\leq \frac{\sqrt{K}}{8m_2e^{w_2 T}}(4c)^{|\alpha|}(2\mathbb{N})^{(p+2)\alpha}.
\end{align*}

\medskip

Now we can finally prove that $u(t,x;\omega)=\sum_{\alpha\in \mathcal{I}}u_\alpha(t,x)H_\alpha(\omega)\in C([0,T],\Hzzeta)\otimes (S)_{-1}.$ Denote by $H=\frac{\sqrt{K}}{8m_2e^{w_2 T}}.$ Then,
\begin{align*}
	\sum_{\alpha\in\mathcal{I}}\sup_{t\in[0,T]}&\|u_\alpha(t)\|^2_{\Hzzeta}(2\mathbb{N})^{-q\alpha}=\sup_{t\in[0,T]}\|u_{\mathbf{0}}(t)\|^2_{\Hzzeta}+\sum_{\alpha>\mathbf{0}}\sup_{t\in[0,T]}\|u_\alpha(t)\|^2_{\Hzzeta}(2\mathbb{N})^{-q\alpha}\\
	&=M_2^2+\sum_{k\in\mathbb{N}}L_{\varepsilon_k}^2(2\mathbb{N})^{-q\varepsilon_k}+\sum_{|\alpha|>1}L_\alpha^2(2\mathbb{N})^{-q\alpha}\\
	&\leq M_2^2+(m_2e^{w_2 T}K)^2\sum_{k\in\mathbb{N}}(2\mathbb{N})^{(2p-q)\varepsilon_k}+H^2\sum_{|\alpha|>1}\Big((4c)^{|\alpha|}(2\mathbb{N})^{(p+2)\alpha}\Big)^2(2\mathbb{N})^{-q\alpha}\\
	&=M_2^2+(m_2e^{w_2 T}K)^2\sum_{k\in\mathbb{N}}(2\mathbb{N})^{(2p-q)\varepsilon_k}+H^2\sum_{|\alpha|>1}(16c^2)^{|\alpha|}(2\mathbb{N})^{(2p+4-q)\alpha}.
\end{align*}
Let $s>0$ be such that $2^s\geq 16c^2.$ According to Lemma \ref{lema ocena}, we obtain
\begin{align*}
	\sum_{\alpha\in\mathcal{I}}\sup_{t\in[0,T]}\|u_\alpha(t)\|^2_X(2\mathbb{N})^{-q\alpha}&\leq M_2^2+(m_2e^{w_2 T}K)^2\sum_{k\in\mathbb{N}}(2\mathbb{N})^{(2p-q)\varepsilon_k}\\
	&+H^2\sum_{|\alpha|>1}(2\mathbb{N})^{(2p+4+s-q)\alpha}<\infty
\end{align*}
for $q>2p+s+5.$
This means that the solution is indeed in $C([0,T],\Hzzeta)\otimes (S)_{-1,-q}$ for all $q>2p+s+5.$
\end{proof}

\appendix

%
%
%
\section{White noise analysis}\label{sec:prel}

The materials in this section mostly come, in a somehow shortened form, from \cite{CPS1}.

%
\subsection{Chaos expansions and the Wick product}\label{subs:wpdsoce}

Denote by $(\Omega, \mathcal{F}, P) $ the Gaussian white noise
probability space $(S'(\mathbb{R}), \mathcal{B}, \mu), $ where
$S'(\mathbb{R})$ denotes the space of tempered distributions,
$\mathcal{B}$ the Borel sigma-algebra generated by the weak topology
on $S'(\mathbb{R})$ and $\mu$ the Gaussian white noise measure corresponding to  the
characteristic function
\begin{equation*}\label{BM theorem}
\int_{S'(\mathbb{R})} \,  e^{{i\langle\omega, \phi\rangle}}
d\mu(\omega) = \exp \left [-\frac{1}{2} \|\phi\|^2_{L^2(\mathbb{R})}\right], \quad \quad\phi\in  S(\mathbb{R}),
\end{equation*}
given by the Bochner-Minlos theorem.

We recall the notions related to $L^2(\Omega,\mu)$ (see \cite{HOUZ}),
where $\Omega=S'(\R)$ and $\mu$ is Gaussian white noise measure. We adopt the notation
$\mathbb N_0=\{0,1,2,\dots\}$, $\mathbb N=\mathbb N_0\setminus\{0\}=\{1,2,\dots\}$.
Define the set of multi-indices $\mathcal I$ to be $(\mathbb N_0^\mathbb N)_c$,
that is, the set of sequences of non-negative integers which have only
finitely many nonzero components. Especially, we denote by $\mathbf 0=(0,0,0,\ldots)$ the multi-index with all entries equal to zero. The length of a multi-index is $|\alpha|=\sum_{i=1}^\infty\alpha_i$ for $\alpha=(\alpha_1,\alpha_2,\ldots)\in\mathcal I$, and it is always finite. Similarly, $\alpha!=\prod_{i=1}^\infty\alpha_i!$, and all other operations are also carried out componentwise. We will use the convention that $\alpha-\beta$ is defined if $\alpha_n-\beta_n\geq0$ for all $n\in\mathbb N$, that is, if $\alpha-\beta\geq\mathbf 0$, and leave $\alpha-\beta$ undefined if $\alpha_n<\beta_n$ for some $n\in\mathbb N$. We here denote by $h_n$, $n\in\mathbb N_0$, the Hermite orthogonal polynomials
\[
h_n(x)=(-1)^n \, e^\frac{x^2}{2}\,\frac{d^n}{dx^n}\left(e^{-\frac{x^2}{2}}\right),
\]
and by $\xi_n$, $n\in\mathbb N$, the Hermite functions
\[
	\xi_n(x)=((n-1)! \sqrt{\pi})^{-\frac{1}{2}}e^{-\frac{x^2}{2}}h_{n-1}(x\sqrt 2).
\]

The
Wiener-It\^o theorem  states that one can define an orthogonal
basis $\{H_\alpha\}_{\alpha\in\mathcal I}$ of $L^2(\Omega,\mu)$,
where $H_\alpha$ are constructed by  means of Hermite orthogonal
polynomials $h_n$ and Hermite functions $\xi_n$,
\begin{equation}\label{Hermitbase}H_\alpha(\omega)=\prod_{n=1} ^\infty h_{\alpha_n}(\langle\omega,\xi_n\rangle),\quad \alpha=(\alpha_1,\alpha_2,\ldots, \alpha_n\ldots)\in\mathcal I,\quad \omega\in\Omega=S'(\mathbb{R}).\end{equation}
Then, every $F\in L^2(\Omega,\mu)$ can be represented via the so
called \emph{chaos expansion}
$$F(\omega)=\sum_{\alpha\in\mathcal I} f_\alpha H_\alpha(\omega), \quad \omega\in S'(\mathbb{R}),\quad\sum_{\alpha\in\mathcal I} |f_\alpha|^2\alpha!<\infty,\quad f_\alpha\in\R,\quad\alpha\in\mathcal I.$$

Denote by $\varepsilon_k=(0,0,\ldots, 1, 0,0,\ldots),\;k\in \N$ the
multi-index with the entry 1 at the $k$th place. Denote by $\mathcal
H_1$ the subspace of $L^2(\Omega,\mu)$, spanned by the polynomials
$H_{\varepsilon_k}(\cdot)$, $k\in\mathbb N$. All elements of  $\mathcal
H_1$ are Gaussian stochastic processes, e.g. the most prominent one is Brownian motion 
given by the chaos expansion $B(t,\omega) = \sum_{k=1}^\infty
\int_0^t \xi_k(s)ds\;H_{\varepsilon_k}(\omega).$

Denote by $\mathcal H_m$ the $m$th order chaos space, that is, the
closure of the linear subspace spanned by the orthogonal polynomials
$H_\alpha(\cdot)$ with $|\alpha|=m$, $m\in\mathbb N_0$. Then the
Wiener-It\^o chaos expansion states that
$L^2(\Omega,\mu)=\bigoplus_{m=0}^\infty \mathcal H_m$, where
$\mathcal H_0$ is the set of constants in $L^2(\Omega,\mu)$. The expectation of a random variable is its orthogonal projection onto $\mathcal H_0$, hence it is given by $E(F(\omega))=f_{\minull}$.


It is well-known that the time-derivative of Brownian motion (white
noise process) does not exist in the classical sense. However,
changing the topology on $L^2(\Omega,\mu)$ to a weaker one, T. Hida
\cite{Hida} defined spaces of generalized random variables
containing the white noise as a weak derivative of the Brownian
motion. We refer to \cite{Hida,HOUZ,HKuo} for white noise
analysis (as an infinite dimensional analogue of the Schwartz theory
of deterministic generalized functions).

Let $(2\mathbb N)^{\alpha}=\prod_{n=1}^\infty (2n)^{\alpha_n},\quad
\alpha=(\alpha_1,\alpha_2,\ldots, \alpha_n,\ldots)\in\mathcal I.$ We
often use the fact that the series $\sum_{\alpha\in\mathcal
	I}(2\mathbb N)^{-p\alpha}$ converges  for $p>1$ \cite[Proposition 2.3.3]{HOUZ}. Define the Banach
spaces
$$(S)_{1,p} =\{F=\sum_{\alpha\in\mathcal I}f_\alpha {H_\alpha}\in L^2(\Omega,\mu):\;  \|F\|^2_{(S)_{1,p}}= \sum_{\alpha\in\mathcal I}(\alpha!)^2 |f_\alpha|^2(2\mathbb N)^{p\alpha}<\infty\},\quad p\in\mathbb N_0.$$
Their topological dual spaces are given by
$$(S)_{-1,-p} =\{F=\sum_{\alpha\in\mathcal I}f_\alpha {H_\alpha}:\;  \|F\|^2_{(S)_{-1,-p}}= \sum_{\alpha\in\mathcal I}|f_\alpha|^2(2\mathbb N)^{-p\alpha}<\infty\},\quad p\in\mathbb N_0.$$
The Kondratiev space of generalized random variables is $(S)_{-1}
=\bigcup_{p\in\mathbb N_0}(S)_{-1,-p}$ endowed with the inductive
topology. It is the strong dual of $(S)_{1} =\bigcap_{p\in\mathbb
	N_0}(S)_{1,p}$, called the Kondratiev space of test random variables
which is endowed with the projective topology.
Thus,
$$(S)_{1} \subseteq L^2(\Omega,\mu) \subseteq
(S)_{-1}$$ forms a Gelfand triplet.

The time-derivative of the Brownian motion exists in the generalized
sense and belongs to the Kondratiev space $(S)_{-1,-p}$ for
$p>\frac5{12}$ \cite[page 21]{HKuo}. We refer to it as to \emph{white noise}
and its formal expansion is given by
$W(t,\omega) =
\sum_{k=1}^\infty \xi_k(t)H_{\varepsilon_k}(\omega).$

In \cite{GRPW}, the definition of stochastic processes is extended
also to processes of the chaos expansion form
$U(t,\omega)=\sum_{\alpha\in\mathcal I}u_\alpha(t)
{H_\alpha}(\omega)$, where the coefficients $u_\alpha$ are elements
of some Banach space $X$. We say that $U$ is an \emph{$X$-valued generalized stochastic process}, that is, $U(t,\omega)\in X\otimes (S)_{-1}$ if there exists $p>0$ such that $\|U\|_{X\otimes(S)_{-1,-p}}^2=\sum_{\alpha\in\mathcal I}\|u_\alpha\|_X^2(2\mathbb N)^{-p\alpha}<\infty$.

The notation $\otimes$ is used for the completion of a tensor product with respect to the $\pi-$topology (see \cite{treves}). We note that  if one of the spaces involved in the tensor product is nuclear, then the completions with respect to the $\pi-$ and the $\varepsilon-$topology coincide. It is known that $(S)_1$ and $(S)_{-1}$ are nuclear spaces \cite[Lemma 2.8.2]{HOUZ}, thus in all forthcoming identities $\otimes$ can be equivalently interpreted 
as the $\widehat{\otimes}_\pi$- or $\widehat{\otimes}_\varepsilon$-completed tensor product. Thus, when dealing 
with the tensor products with $(S)_{1,p}$ and $(S)_{-1,-p}$, we work with the $\pi$-topology.

The \emph{Wick product} of two stochastic
processes $F=\sum_{\alpha\in\mathcal I}f_\alpha H_\alpha$ and
$G=\sum_{\beta\in\mathcal I}g_\beta H_\beta\in X\otimes(S)_{-1}$
is given by $$F\lozenge G = \sum_{\gamma\in\mathcal
	I}\sum_{\alpha+\beta=\gamma}f_\alpha g_\beta H_\gamma =
\sum_{\alpha\in\mathcal I}\sum_{\beta\leq \alpha} f_\beta
g_{\alpha-\beta} H_\alpha,$$ and the $n$th Wick power is defined by
$F^{\lozenge n}=F^{\lozenge (n-1)}\lozenge F$, $F^{\lozenge 0}=1$.
Note that
$H_{n\varepsilon_k}=H_{\varepsilon_k}^{\lozenge n}$ for $n\in\mathbb
N_0$, $k\in\mathbb N$. The Wick product always exists and results in a new element of $X\otimes(S)_{-1}$, moreover it exhibits the property of $E(F\lozenge G)=E(F)E(G)$ holding true. The ordinary product of two generalized stochastic processes does not always exist and $E(F\cdot G)=E(F)E(G)$ would hold only if $F$ and $G$ were uncorrelated.

One particularly important choice for the Banach space $X$ is $X=C^k[0,T]$, $k\in\mathbb N$.
In \cite{ps} it is proved that differentiation of a stochastic process can be carried
out componentwise in the chaos expansion, that is, due to the fact that
$(S)_{-1}$ is a nuclear space it holds that
$C^k([0,T],(S)_{-1})=C^k[0,T]\otimes(S)_{-1}$. This means that a
stochastic process $U(t,\omega)$ is $k$ times continuously
differentiable if and only if all of its coefficients $u_\alpha(t)$,
$\alpha\in\mathcal I$ are in $C^k[0,T]$.

The same holds for Banach space valued stochastic processes that is,
elements of  $C^k([0,T],X)\otimes(S)_{-1}$, where $X$ is an
arbitrary Banach space. By the nuclearity of $(S)_{-1}$, these
processes can be regarded as elements of the tensor product spaces
$$C^k([0,T],X\otimes
(S)_{-1})=C^k([0,T],X)\otimes(S)_{-1}=\bigcup_{p=0}^{\infty}C^k([0,T],X)\otimes
(S)_{-1,-p}.$$

In order to solve \eqref{eq:SPDE} we choose some specific Banach spaces, suggested by the associated deterministic theory.
In general, the function spaces that we will adopt as those where to
look for the solutions to \eqref{eq:SPDE} will be of the form
\begin{equation}\label{eq:solspa}
	L^2(I,G_k)\otimes (S)_{-1}, \quad k\in\mathbb Z,
\end{equation}
or
\begin{equation}\label{eq:solspb}
	\bigcap_{l\ge k\ge0} C^k(I,G_k)\otimes (S)_{-1},\quad 1\le l\le \infty,
\end{equation}
where $I\subset\R$ is an interval of the form $[0,T]$ or $[0,\infty)$,
and $G_k$, $k=0,1,2,\cdots,l$,
or $k\in\Z_+$, are suitable Hilbert spaces (or Banach spaces) such that
\[
\cdots\hookrightarrow G_{k+1}\hookrightarrow G_{k}\cdots \hookrightarrow G_1 \hookrightarrow G_{0},
\]
where $\hookrightarrow$ denotes dense continuous embeddings. We can also consider the topological duals of $G_j$, $j\in\mathbb Z_+$, denoted by $G_{-j}$, respectively, and write
\[
G_{0}\hookrightarrow G_{-1} \hookrightarrow G_{-2}\hookrightarrow\cdots\hookrightarrow G_{-k}\hookrightarrow G_{-(k+1)}\hookrightarrow\cdots.
\]
%

In particular, for the spaces in \eqref{eq:solspa} and
in \eqref{eq:solspb} 
we have, respectively,
\[
L^2(I,G_k)\otimes (S)_{-1} \simeq L^2(I,G_k\otimes (S)_{-1}) \simeq \bigcup_{r=0}^\infty L^2(I,G_k)\otimes(S)_{-1,-r},
\]
\[
C^j(I,G_k)\otimes (S)_{-1} \simeq C^j(I,G_k\otimes (S)_{-1}) \simeq \bigcup_{r=0}^\infty C^j(I,G_k)\otimes(S)_{-1,-r}.
\]

\subsection{Estimates on functions of multiindeces} \label{subs:catalan}
We also recall some useful estimates that we intensely utilize in Section \ref{sec:wickp}. The proofs of these estimates can be found in \cite{KL} and \cite{Milica2}.
\begin{lemma} \label{ocena faktorijela} Let $\alpha\in \mathcal{I}.$ Then,
	\begin{align*}
		\frac{|\alpha|!}{\alpha!}&\leq (2\mathbb{N})^{2\alpha}.
	\end{align*}
\end{lemma}

\begin{lemma} \label{lema ocena} For every $c>0$ there exists $q>1$ such that
	\begin{align*}
		\sum_{\alpha\in\mathcal{I}}c^{|\alpha|}(2\mathbb{N})^{-q\alpha}<\infty.
	\end{align*}
\end{lemma}
\begin{lemma}\label{katalan}
	A sequence $\{\mathbf{c}_n\}_{n\in \mathbb N}$  defined by the recurrence relation
	\begin{equation}
		\label{Catalan}
		\mathbf{c}_0=1, \quad 
		\mathbf{c}_n=  \sum_{k=0}^{n-1} \, \mathbf{c}_k  \, \mathbf{c}_{n-1-k},  
		\quad n\geq 1 ,
	\end{equation}
	is called  the sequence of Catalan numbers.  The closed formula for $\mathbf{c}_n$ is a multiple of the binomial coefficient, that is,  the solution of the Catalan recurrence \eqref{Catalan} is 
	\begin{equation*}\label{Catalan1}
		\mathbf{c}_n= \frac1{n+1} \, \binom{2n}{n}\quad\mbox{or}\quad \mathbf{c}_n = \binom{2n}{n} - \binom{2n}{n+1}.
	\end{equation*}
	The Catalan numbers satisfy the growth estimate
	\begin{equation}\label{Cat ocena}
		\mathbf{c}_n\leq 4^n,\;n\geq 0.
	\end{equation}
\end{lemma}
\begin{lemma} \label{multi katalan}(\cite[p.21]{KL})
	Let $\{R_\alpha:\;\alpha\in \mathcal{I}\}$ be a set of real numbers such that $R_\mathbf{0}=0,\;R_{\varepsilon_k},\; k\in\mathbb{N}$, are given, and 
	$$R_\alpha=\sum_{\mathbf{0}<\gamma<\alpha}R_\gamma R_{\alpha-\gamma},\quad |\alpha|>1.$$
	Then,
	$$R_\alpha=\frac{1}{|\alpha|}\binom{2|\alpha|-2}{|\alpha|-1} \frac{|\alpha|!}{\alpha!}\prod_{k=1}^\infty R_{\varepsilon_k}^{\alpha_k},\quad |\alpha|>1.$$
\end{lemma}

%
\subsection{Stochastic operators and differential operators with stochastic coefficients}\label{subs:dwops}
Let $X$ be a Banach space endowed with the norm $\|\cdot\|_X$. Consider $X\otimes (S)_{-1}$ with elements $u=\sum_{\alpha\in\mathcal I}u_\alpha H_\alpha$ so that  $\sum_{\alpha\in\mathcal I}\|u_\alpha\|^2_X(2\mathbb N)^{-p\alpha}<\infty$ for some $p\geq 0$. Let $D\subset X$ be a dense subset of $X$ endowed with the norm $\|\cdot\|_D$ and  $A_\alpha:D\rightarrow X$, $\alpha\in\mathcal I$, be a family of linear operators on this common domain $D$. Assume that each $A_\alpha$ is bounded that is, $$\|A_\alpha\|_{\caL(D,X)}=\mathrm{sup}\{\|A_\alpha(x)\|_X:\; \|x\|_D\leq 1\}<\infty.$$ In case when $D=X$, we will write $\caL(X)$ instead of $\caL(D,X)$.

The family of operators $A_\alpha$, $\alpha\in\mathcal I$, gives rise to a stochastic operator $\mathbf A\lozenge:D\otimes(S)_{-1}\rightarrow X\otimes (S)_{-1}$, that acts in the following manner
\[
\mathbf A\lozenge u = \sum_{\gamma\in\mathcal I}\left(\sum_{\beta+\lambda=\gamma} A_\beta(u_\lambda)\right)H_\gamma.
\]
In the next two lemmas we provide two sufficient conditions that ensure the stochastic operator $\mathbf A\lozenge$ to be well-defined. Both conditions rely on the $l^2$ or $l^1$ bounds with suitable weights. They are actually equivalent to the fact that $A_\alpha$, $\alpha\in\mathcal I$, are polynomially bounded, but they provide finer estimates on the stochastic order (Kondratiev weight) of the domain and codomain of $\mathbf A\lozenge$. Their proofs can be found in \cite{CPS1}.
\begin{lemma}\label{doralemma1}
	If the operators $A_\alpha$, $\alpha\in\mathcal I$,  satisfy $\sum_{\alpha\in\mathcal
		I}\|A_\alpha\|^2_{\caL(D,X)}(2\mathbb N)^{-r\alpha}<\infty$, for some $r\geq0$, then
	$\mathbf A\lozenge$ is well-defined as a mapping $\mathbf A\lozenge: D\otimes(S)_{-1,-p}\rightarrow X\otimes(S)_{-1,-(p+r+m)}$, $m>1$.
\end{lemma}
\begin{lemma}\label{doralemma2}
	If the operators $A_\alpha$, $\alpha\in\mathcal I$,  satisfy  $\sum_{\alpha\in\mathcal
		I}\|A_\alpha\|_{\caL(D,X)}(2\mathbb N)^{-\frac{r}{2}\alpha}<\infty$, for some $r\geq0$, then $\mathbf A\lozenge$ is well-defined as a mapping $\mathbf A\lozenge: D\otimes(S)_{-1,-r}\rightarrow X\otimes(S)_{-1,-r}$.
\end{lemma}
For example, let $D=H_0^1(\mathbb R)$, $X=L^2(\mathbb R)$ and $A_\alpha = a_\alpha\cdot \partial_x$, $a_\alpha\in\mathbb R$, be scalars such that $\sum_{\alpha\in\mathcal
	I}|a_\alpha|^2(2\mathbb N)^{-r\alpha}<\infty$, for some $r\geq0$. Then $\|A_\alpha\|_{\caL(D,X)}=|a_\alpha|$, hence for $u\in H_0^1(\mathbb R)\otimes (S)_{-1}$ we have 
	\[
		\mathbf A\lozenge u(x,\omega)=\sum_{\gamma\in\mathcal I}\left(\sum_{\alpha+\beta=\gamma}a_\alpha\cdot \partial_x(u_\beta(x))\right)H_\gamma(\omega)
	\]
	is a well-defined element in $L^2(\mathbb R)\otimes(S)_{-1}$. A similar example may be constructed with $D=L^2(\mathbb R)$ and $X=H^{-1}(\mathbb R)$. Note that in these examples, we could have written the operator also in the form $\mathbf A=a(\omega)\partial_x$, where $a(\omega)=\sum_{\alpha\in\mathcal I}a_\alpha H_\alpha(\omega)\in (S)_{-1,-r}$.

Considering the differential operator $\mathbf L$ that governs equation \eqref{eq:SPDE}, we have made special choices for the domain $D$ and range $X$, involving (subspaces of) the (weighted) Sobolev-Kato spaces $H^{z,\zeta}(\mathbb R^d)$ and many other types of spaces that stem from the $SG$ pseudodifferential calculus. 

\section{The calculus of $SG$ pseudodifferential operators}\label{subs:sgcalc}
%
%
%
We here recall some basic definitions and facts about the $SG$-calculus of pseudodifferential 
operators, through
standard material appeared, e.g., in \cite{ACS19b,CPS1} and elsewhere (sometimes with slightly different notational choices).
We often employ the so-called \textit{japanese bracket} of $y\in\R^d$, given by $\langle y \rangle =\sqrt{1+|y|^2}$.

The class $S ^{m,\mu}=S ^{m,\mu}(\R^{d})$ of $SG$ symbols of order $(m,\mu) \in \R^2$ is given by all the functions 
$a(x,\xi) \in C^\infty(\R^d\times\R^d)$
with the property
that, for any multiindices $\alpha,\beta \in \N_0^d$, there exist
constants $C_{\alpha\beta}>0$ such that the conditions 
\begin{equation}
	\label{eq:disSG}
	|D_x^{\alpha} D_\xi^{\beta} a(x, \xi)| \leq C_{\alpha\beta} 
	\x^{m-|\alpha|}\csi^{\mu-|\beta|},
	\qquad (x, \xi) \in \R^d \times \R^d,
\end{equation}
hold (see \cite{cordes,ME,PA72}). We often omit the base spaces $\R^d$, $\R^{2d}$, etc., from the notation.

For $m,\mu\in\R$, $\ell\in\N_0$,
\[
	\vvvert a \vvvert^{m,\mu}_\ell
	= 
	\max_{|\alpha+\beta|\le \ell}\sup_{x,\xi\in\R^d}\x^{-m+|\alpha|} 
	                                                                     \csi^{-\mu+|\beta|}
	                                                                    | \partial^\alpha_x\partial^\beta_\xi a(x,\xi)|, \quad a\in\ S^{m,\mu},
\]
is a family of seminorms, defining  the Fr\'echet topology of $S^{m,\mu}$.

The corresponding
classes of pseudodifferential operators $\Op (S ^{m,\mu})=\Op (S ^{m,\mu}(\R^d))$ are given by
\begin{equation}\label{eq:psidos}
	(\Op(a)u)(x)=(a(.,D)u)(x)=(2\pi)^{-d}\int e^{\ii x\xi}a(x,\xi)\hat{u}(\xi)d\xi, \quad a\in S^{m,\mu}(\R^d),u\in\caS(\R^d),
\end{equation}
extended by duality to $\caS^\prime(\R^d)$.
The operators in \eqref{eq:psidos} form a
graded algebra with respect to composition, that is,
$$
\Op (S ^{m_1,\mu _1})\circ \Op (S ^{m_2,\mu _2})
\subseteq \Op (S ^{m_1+m_2,\mu _1+\mu _2}).
$$
The symbol $c\in S ^{m_1+m_2,\mu _1+\mu _2}$ of the composed operator $\Op(a)\circ\Op(b)$,
$a\in S ^{m_1,\mu _1}$, $b\in S ^{m_2,\mu _2}$, admits the asymptotic expansion
\begin{equation}
	\label{eq:comp}
	c(x,\xi)\sim \sum_{\alpha}\frac{i^{|\alpha|}}{\alpha!}\,D^\alpha_\xi a(x,\xi)\, D^\alpha_x b(x,\xi),
\end{equation}
which implies that the symbol $c$ equals $a\cdot b$ modulo $S ^{m_1+m_2-1,\mu _1+\mu _2-1}$.

Note that
\[
	 S ^{-\infty,-\infty}=S ^{-\infty,-\infty}(\R^{d})= \bigcap_{(m,\mu) \in \R^2} S ^{m,\mu} (\R^{d})
	 =\caS(\R^{2d}).
\]
For any $a\in S^{m,\mu}$, $(m,\mu)\in\R^2$,
$\Op(a)$ is a linear continuous operator from $\caS(\R^d)$ to itself, extending to a linear
continuous operator from $\caS^\prime(\R^d)$ to itself, and from
$H^{s,\sigma}(\R^d)$ to $H^{s-m,\sigma-\mu}(\R^d)$,
where $H^{s,\sigma}=H^{s,\sigma}(\R^d)$,
$(s,\sigma) \in \R^2$, denotes the  Sobolev-Kato (or \textit{weighted Sobolev}) space
\begin{equation}\label{eq:skspace}
  	H^{s,\sigma}(\R^d)= \{u \in \caS^\prime(\R^{n}) \colon \|u\|_{s,\sigma}=
	\|{\jap}^s\pdd^\sigma u\|_{L^2}< \infty\},
\end{equation}
(here $\pdd^\sigma$ is understood as a pseudodifferential operator) with the naturally induced Hilbert norm. When $s\ge s^\prime$ and $\sigma\ge\sigma^\prime$, the continuous embedding 
$H^{s,\sigma}\hookrightarrow H^{s^\prime,\sigma^\prime}$ holds true. It is compact when $s>s^\prime$ and $\sigma>\sigma^\prime$.
Since $H^{s,\sigma}=\jap^{-s}\,H^{0,\sigma}=\jap^{-s}\, H^\sigma$, with $H^\sigma$ the usual Sobolev space of order $\sigma\in\R$, we 
find $\sigma>k+\dfrac{d}{2} \Rightarrow H^{s,\sigma}\hookrightarrow C^k(\R^d)$, $k\in\N_0$. One actually finds
\begin{equation}\label{eq:spdecomp}
	\bigcap_{s,\sigma\in\R}H^{s,\sigma}(\R^d)=H^{\infty,\infty}(\R^d)=\caS(\R^d),
	\quad
	\bigcup_{s,\sigma\in\R}H^{s,\sigma}(\R^d)=H^{-\infty,-\infty}(\R^d)=\caS^\prime(\R^d),
\end{equation}
as well as, for the space of \textit{rapidly decreasing distributions}, see \cite[Chap. VII, \S 5]{schwartz}, 
\begin{equation}\label{eq:rdd}
	\caS^\prime(\R^d)_\infty=\bigcap_{s\in\R}\bigcup_{\sigma\in\R}H^{s,\sigma}(\R^d)=H^{\infty,-\infty}(\R^d).
\end{equation}
The continuity property of
the elements of $\Op(S^{m,\mu})$ on the scale of spaces $H^{s,\sigma}(\R^d)$, $(m,\mu),(s,\sigma)\in\R^2$, is expressed 
more precisely in the next theorem.
\begin{theorem}[{\cite[Chap. 3, Theorem 1.1]{cordes}}] \label{thm:sobcont}
	Let $a\in S^{m,\mu}(\R^d)$, $(m,\mu)\in\R^2$. Then, for any $(s,\sigma)\in\R^2$, 
	$\Op(a)\in\mathcal{L}(H^{s,\sigma}(\R^d),H^{s-m,\sigma-\mu}(\R^d))$, and there exists a constant $C>0$,
	depending only on $d,m,\mu,s,\sigma$, such that
	\begin{equation}\label{eq:normsob}
		\|\Op(a)\|_{\scrL(H^{s,\sigma}(\R^d), H^{s-m,\sigma-\mu}(\R^d))}\le 
		C\vvvert a \vvvert_{\left[\frac{d}{2}\right]+1}^{m,\mu},
	\end{equation}
	where $[t]$ denotes the integer part of $t\in\R$.
\end{theorem}
The class $\caO(m,\mu)$ of the \textit{operators of order $(m,\mu)$} is introduced as follows, see, e.g., \cite[Chap. 3, \S 3]{cordes}.
\begin{definition}\label{def:ordmmuopr}
	A linear continuous operator $A\colon\caS(\R^d)\to\caS(\R^d)$
	belongs to the class $\caO(m,\mu)$, $(m,\mu)\in\R^2$, of the operators of order $(m,\mu)$ if, for any $(s,\sigma)\in\R^2$,
	it extends to a linear continuous operator $A_{s,\sigma}\colon H^{s,\sigma}(\R^d)\to H^{s-m,\sigma-\mu}(\R^d)$. We also define
	\[
		\caO(\infty,\infty)=\bigcup_{(m,\mu)\in\R^2} \caO(m,\mu), \quad
		\caO(-\infty,-\infty)=\bigcap_{(m,\mu)\in\R^2} \caO(m,\mu).		
	\]
\end{definition}
\begin{remark}\label{rem:O}
	\begin{enumerate}
		\item[(i)] Trivially, any $A\in\caO(m,\mu)$ admits a linear continuous extension 
		$A_{\infty,\infty}\colon\caS^\prime(\R^d)\to\caS^\prime(\R^d)$. In fact, in view of \eqref{eq:spdecomp}, it is enough to set
		$A_{\infty,\infty}|_{H^{s,\sigma}(\R^d)}= A_{s,\sigma}$.
		\item[(ii)] Theorem \ref{thm:sobcont} implies $\Op(S^{m,\mu}(\R^d))\subset\caO(m,\mu)$, $(m,\mu)\in\R^2$.
		\item[(iii)] $\caO(\infty,\infty)$ and $\caO(0,0)$ are algebras under operator multiplication, $\caO(-\infty,-\infty)$ is an ideal
		of both  $\caO(\infty,\infty)$ and $\caO(0,0)$, and 
		$\caO(m_1,\mu_1)\circ\caO(m_2,\mu_2)\subset\caO(m_1+m_2,\mu_1+\mu_2)$.
	\end{enumerate}
\end{remark}
\noindent
The following characterization of the class $\caO(-\infty,-\infty)$ is often useful.
\begin{proposition}[{\cite[Ch. 3, Prop. 3.4]{cordes}}] \label{thm:smoothing}
	The class $\caO(-\infty,-\infty)$ coincides with $\Op(S^{-\infty,-\infty}(\R^d))$ and with the class of smoothing operators,
	that is, the set of all the linear continuous operators $A\colon\caS^\prime(\R^d)\to\caS(\R^d)$. All of them coincide with the
	class of linear continuous operators $A$ admitting a Schwartz kernel $k_A$ belonging to $\caS(\R^{2d})$. 
\end{proposition}
An operator $A=\Op(a)$ and its symbol $a\in S ^{m,\mu}$ are called \emph{elliptic}
(or $S ^{m,\mu}$-\emph{elliptic}) if there exists $R\ge0$ such that
%
\[
	C\x^{m} \csi^{\mu}\le |a(x,\xi)|,\qquad 
	|x|+|\xi|\ge R,
\] 
for some constant $C>0$. If $R=0$, $a^{-1}$ is everywhere well-defined and smooth, and $a^{-1}\in S ^{-m,-\mu}$.
If $R>0$, then $a^{-1}$ can be extended to the whole of $\R^{2d}$ so that the extension $\widetilde{a}_{-1}$ satisfies $\widetilde{a}_{-1}\in S ^{-m,-\mu}$.
An elliptic $SG$ operator $A \in \Op (S ^{m,\mu})$ admits a
parametrix $A_{-1}\in \Op (S ^{-m,-\mu})$ such that
\[
A_{-1}A=I + R_1, \quad AA_{-1}= I+ R_2,
\]
for suitable $R_1, R_2\in\Op(S^{-\infty,-\infty})$, where $I$ denotes the identity operator. 
In such a case, $A$ turns out to be a Fredholm
operator on the scale of functional spaces $H^{s,\sigma}$,
$(s,\sigma)\in\R^2$.

\end{document}